\documentclass{amsart}
\usepackage{latexsym,amsxtra,amscd,ifthen}
\usepackage{amsfonts, color}
\usepackage{verbatim}
\usepackage{amsmath}
\usepackage{amsthm}
\usepackage{amssymb}
\usepackage{url}
\usepackage[arrow,matrix]{xy}

\theoremstyle{plain}

\newtheorem{theorem}{Theorem}
\newtheorem{lemma}[theorem]{Lemma}
\newtheorem{proposition}[theorem]{Proposition}
\newtheorem{corollary}[theorem]{Corollary}
\newtheorem{conjecture}[theorem]{Conjecture}
\newtheorem{hypothesis}[theorem]{Hypothesis}

\numberwithin{theorem}{section}
\numberwithin{equation}{theorem}

\theoremstyle{definition}
\newtheorem{definition}[theorem]{Definition}

\newtheorem{example}[theorem]{Example}
\newtheorem{remark}[theorem]{Remark}
\newtheorem{question}[theorem]{Question}
\newtheorem*{question*}{Question}

\DeclareMathOperator{\tr}{tr}

\DeclareMathOperator{\injdim}{injdim}

\DeclareMathOperator{\End}{End}
\DeclareMathOperator{\Ext}{Ext}
\DeclareMathOperator{\Tor}{Tor}

\DeclareMathOperator{\GKdim}{GKdim}
\DeclareMathOperator{\Kdim}{Kdim}

\DeclareMathOperator{\Aut}{Aut}
\DeclareMathOperator{\gr}{gr}
\DeclareMathOperator{\Mod}{Mod}

\def\Rpf{{\sf r}}
\def\Pty{{\sf p}}

\def\mmod{\operatorname{mod}}

\def\qmod{\operatorname{qmod}}

\newcommand{\inth}{\textstyle \int}

\begin{document}

\title{Noncommutative Auslander theorem}

\author{Y.-H. Bao, J.-W. He and J. J. Zhang}

\address{Bao: School of Mathematical Sciences,
Anhui University, Hefei, Anhui, 230601, China}

\email{baoyh@ahu.edu.cn}

\address{He: Department of Mathematics,
Hangzhou Normal University,
Hangzhou Zhejiang 310036, China}
\email{jwhe@hznu.edu.cn}

\address{Zhang: Department of Mathematics, Box 354350,
University of Washington, Seattle, Washington 98195, USA}

\email{zhang@math.washington.edu}

\begin{abstract}
In 1960s Maurice Auslander proved the following important result.
Let $R$ be the commutative polynomial ring $\mathbb{C}[x_1,\dots,x_n]$, 
and let $G$ be a finite small subgroup of $GL_n(\mathbb{C})$ acting on 
$R$ naturally. Let $A$ be the fixed subring $R^G:=\{a\in R|g(a)=a,
\text{ for all } g\in G \}$. Then the endomorphism ring of the right $A$-module 
$R_A$ is naturally isomorphic to the skew group algebra $R\ast G$. In this paper, 
a version of Auslander Theorem is proven for the following classes of 
noncommutative algebras: (a) noetherian PI local (or connected graded)
algebras of finite injective dimension,
(b) universal enveloping algebras of finite
dimensional Lie algebras, and (c) noetherian graded down-up
algebras.
\end{abstract}





\maketitle


\setcounter{section}{-1}
\section{Introduction}
\label{xxsec0}

The Auslander theorem as stated in the abstract (see also \cite{Au1, Au2} 
or \cite[Theorem 4.2]{IT}) 
is a fundamental result in the study of the McKay correspondence, isolated
singularities, and other homological aspects of commutative algebras.
Recently, several researchers or research groups have been interested in
the Auslander theorem in the noncommutative setting, see
\cite{CKWZ1, HVZ, Ki, Mor, MU, Ue}. Here is a partial list of results concerning
noncommutative versions of the Auslander theorem during the last few years:
\begin{enumerate}
\item[(1)]
Mori-Ueyama proved a version of the Auslander theorem when
the fixed subring has graded isolated singularities \cite{MU}.
\item[(2)]
Van Oystaeyen-Zhang and the second-named author proved a version
the Auslander theorem for $H^*$-dense Galois extensions \cite{HVZ}.
\item[(3)]
Chan-Kirkman-Walton and the third-named author proved a version
of the Auslander theorem for Hopf actions on Artin-Schelter
regular algebras of global dimension two with trivial
homological determinant \cite{CKWZ1}.
\end{enumerate}

More generally the authors proved some results concerning the
Auslander theorem that extends the results in (1,2,3) above,
see \cite{BHZ}.

This paper is a sequel to \cite{BHZ}. We freely use terminologies
introduced in \cite{BHZ}. Our  motivation is to understand
noncommutative McKay correspondence where
one of  important ingredients is the Auslander theorem. Here we apply
a main result of \cite{BHZ} to establish the Auslander
theorem for several different classes of noncommutative algebras.
This research is related to noncommutative algebraic geometry,
noncommutative invariant theory and representation theory of
noncommutative algebras.

The first family of algebras that we are interested in are
noetherian local or connected graded algebras that satisfy a
{\it polynomial identity} (abbreviated as {\it PI}). We
need to recall some definitions before stating the results.

Let $A$ be an algebra over a base field $\Bbbk$ and $H$ a nontrivial
finite dimensional semisimple Hopf algebra acting on $A$ inner faithfully
\cite[Definition 3.9]{BHZ}.  Let $A\# H$ denote
the smash product \cite[Definition 4.1.3]{Mon}. Note that both $A$ and
$H$ are subalgebras of $A\# H$. Since $H$ is semisimple, the left and
the right integrals of $H$ coincide, and is denoted by $\inth$.

Let $\partial$ be a dimension function defined on the right $A$-modules
(or on finitely generated right $A$-modules) in the sense of \cite[6.8.4]{MR}.
In this paper $\partial$ is either the Gelfand-Kirillov
dimension (denoted by $\GKdim$) -- see \cite{KL} and
\cite[Definition 1.1]{BHZ}, or the Krull dimension (denoted by $\Kdim$)
-- see \cite[Chapter 6]{MR}.
Ideas in this paper should apply to other dimension functions.

\begin{definition}
\label{xxdef0.1} Let $\partial$ be a dimension function on
right $A$-modules.
\begin{enumerate}
\item[(1)] \cite[Definition 0.1]{BHZ}
Suppose that $\partial(A)<\infty$. The {\it pertinency}
of the $H$-action on $A$ with respect to $\partial$ is defined to be
$$\Pty_{\partial}(A,H):=\partial(A)-\partial ((A\# H)/I)$$
where $I$ is the 2-sided ideal of $A\# H$ generated by
$1\# \inth$. If $\partial=\GKdim$, then $\Pty_{\partial}(A,H)$
is denoted by $\Pty(A,H)$.
\item[(2)] \cite[Definition 0.2(1)]{BHZ}
The {\it grade} of a right $A$-module $M$ is defined to be
$$j(M) := \min\{i \mid  \Ext^i_A(M,A)\neq 0\}$$
If $\Ext^i_A(M,A)=0$ for all $i$, then we say $j(M)=\infty$.
\item[(3)]\cite[Definition 0.2(2)]{BHZ}
If $j((A\# H)/I)\geq 2$ where $(A\# H)/I$ is viewed as a right
$A$-module, we say the $H$-action on $A$ is
{\it homologically small} or {\it h.small}.
\end{enumerate}
\end{definition}

Pertinency is an essential invariant
introduced in \cite{BHZ}. It is an invariant of the $H$-action on
$A$, not just the pair $(A,H)$.
Recall that, for a finite dimensional
$\Bbbk$-vector space $V$, a finite subgroup
$G\subseteq {\rm{GL}}(V)$ is called {\it small}
if $G$ does not contain a pseudo-reflection.
When $A$ is the polynomial ring and $H$ is a group algebra, then
homological smallness is equivalent to smallness
by Lemma \ref{xxlem7.2}.

Now we are ready to state our first result concerning the Auslander theorem.

\begin{theorem}
\label{xxthm0.2}
Let $R$ be a noetherian, PI local {\rm{(}}or connected
graded{\rm{)}} algebra of finite injective dimension $\geq 2$
and $H$ a semisimple Hopf algebra acting on $R$.
Then the following are equivalent:
\begin{enumerate}
\item[(1)]
There is a natural isomorphism of algebras $R\# H\cong \End_{R^H}(R)$.
\item[(2)]
$\Pty_{\Kdim}(R,H)\geq 2$.
\item[(3)]
The $H$-action on $R$ is homologically small.
\end{enumerate}
\end{theorem}

In the above theorem, even if $A$ is graded, we do not assume that the
$H$-action is homogeneous. When $R$ is the commutative
polynomial ring $\Bbbk[V]$ and $G$ acts on a $\Bbbk$-vector
space $V$ inner faithfully, the above theorem agrees with the original
Auslander theorem stated at the beginning of the paper.
One immediate question is the following.

\begin{question}
\label{xxque0.3}
Does a version of Theorem \ref{xxthm0.2} holds without the PI hypothesis?
\end{question}

Considering an $H$-action on an algebra $A$,
we say that {\it the Auslander theorem holds} if there is a natural
isomorphism of algebras $A\# H \cong \End_{A^H}(A)$. By
Theorem \ref{xxthm0.2}, the pertinency is ultimately connected
with the Auslander theorem. Under some reasonable hypotheses
(such as $A$ being Cohen-Macaulay, Artin-Schelter regular, and
so on), $\Pty(A,H)=0$ if and only if the $H$-action on $A$
is not inner-faithful. Equivalently, $\Pty(A,H)\geq 1$ if and
only if the $H$-action on $A$ is inner-faithful. It follows
from the Auslander theorem that $\Pty(A,H)\geq 2$ if and
only if $A\# H$ is naturally isomorphic to $\End_{A^H}(A)$.
When $\Pty(A,H)$ is maximal possible, namely,
$\Pty(A,H)=\GKdim A$, the fixed ring $A^H$ has isolated
singularities \cite{BHZ}. In general, $\Pty(A,H)$
controls the dimension of the singularities of $A^H$,
see \cite[(E0.5.1)]{BHZ}. Therefore the pertinency is
an important invariant related to several properties
of the $H$-action on $A$.

For different classes of algebras, proofs of the Auslander theorem
are different and sometimes require different technologies.
This is a main reason why we  have  different hypotheses
in different theorems.

By using the idea of mod-$p$ reduction, we are able to prove a
version of Theorem \ref{xxthm0.2} for a class of
non-local noncommutative algebras.

\begin{theorem}
\label{xxthm0.4} Suppose $\rm{char}\; \Bbbk=0$.
Let $R$ be the universal enveloping algebra
$U(\mathfrak g)$ of a finite dimensional
Lie algebra ${\mathfrak g}$. Suppose $G$ is a
finite small subgroup of $\Aut_{Lie}(\mathfrak g)$.
Then there is a natural isomorphism of algebras
$R\ast G\cong \End_{R^G}(R)$.
\end{theorem}

Note that $\Aut_{Lie}({\mathfrak g})$ is the group of
Lie algebra automorphisms of ${\mathfrak g}$.
The following is an immediate consequence of Theorem
\ref{xxthm0.4}.

\begin{corollary}
\label{xxcor0.5} Suppose $\rm{char}\; \Bbbk=0$.
Let $R$ be the universal enveloping algebra
$U(\mathfrak g)$ of a finite dimensional
Lie algebra ${\mathfrak g}$. Suppose that
${\mathfrak g}\neq {\mathfrak g}'\ltimes \Bbbk x$
for a 1-dimensional Lie ideal $\Bbbk x\subseteq {\mathfrak g}$.
Then $R\ast G\cong \End_{R^G}(R)$
for every finite group
$G\subseteq \Aut_{Lie}({\mathfrak g})$.
\end{corollary}

Theorem \ref{xxthm0.4} and Corollary \ref{xxcor0.5} suggest
that there should be a version of the McKay correspondence for
the universal enveloping algebra of a finite dimensional Lie
algebra, which would be a very interesting
future project. Ideas of the proof of Theorem \ref{xxthm0.4} applies
to other algebras with good filtration, see Theorem \ref{xxthm4.10}.

Down-up algebras were introduced by Benkart and Roby in \cite{BR}  as a tool
to study the structure of certain posets. Noetherian graded down-up
algebras are Artin-Schelter regular of global dimension three with two
generators \cite{KMP}. Their graded automorphism groups were computed in
\cite{KK}, and are rich enough to provide many nontrivial examples.
Some invariant theoretic aspects concerning down-up algebras have
been studied in \cite{KK, KKZ1}. We have a version of the
Auslander theorem for down-up algebras. For a graded algebra $A$,
let $\Aut_{gr}(A)$ be the group of all graded algebra automorphisms
of $A$.

\begin{theorem}
\label{xxthm0.6}
Suppose $\rm{char}\; \Bbbk=0$.
Let $R$ be a noetherian graded down-up algebra
$A(\alpha,\beta)$ generated by $V:=\Bbbk x+\Bbbk y$ -see Definition
{\rm{\ref{xxdef6.1}}}. Assume
that either $\beta\neq -1$ or $(\alpha, \beta)=(2,-1)$.
Let $G$ be any nontrivial finite subgroup of $\Aut_{gr}(R)$.
Then the following hold.
\begin{enumerate}
\item[(1)]
$\Pty(R,G)\geq 2$.
\item[(2)]
The $G$-action on $R$ is homologically small.
\item[(3)]
There is a natural isomorphism of graded algebras
$R\ast G\cong\End_{R^G}(R).$
\end{enumerate}
\end{theorem}

Theorem \ref{xxthm0.6} suggests that there should be a version
of the McKay correspondence for noetherian graded down-up
algebra, which would be a good test project for understanding
the noncommutative McKay correspondence in dimension three.

It is unfortunate that we need some extra hypothesis on the
parameters $(\alpha,\beta)$ in Theorem \ref{xxthm0.6}.
Our conjecture is

\begin{conjecture}
\label{xxcon0.7}
Theorem {\rm{\ref{xxthm0.6}}} also holds in the
cases when $\beta=-1$ and $\alpha\neq 2$.
\end{conjecture}

See \cite{CKZ2, GKMW} for some results related to Conjecture
\ref{xxcon0.7}. For completeness we work out a version of
Auslander theorem for skew polynomial rings (see Example \ref{xxex5.2}).
Let $\Bbbk^{\times}$ denote the set of units in $\Bbbk$.

\begin{theorem}
\label{xxthm0.8}
Let $R$ be the skew polynomial ring $\Bbbk_{p_{ij}}[x_1,\cdots,x_n]$
for $n\geq 2$ and $p_{ij}\in \Bbbk^{\times}$ being generic.
Let $G$ be a finite group of algebra automorphisms of $R$. Then
$G$ acts on the finite dimensional $\Bbbk$-space $V:=
\oplus_{s=1}^{n} \Bbbk x_s$ and the following are equivalent.
\begin{enumerate}
\item[(1)]
$G\subseteq {\rm{GL}}(V)$ is small.
\item[(2)]
The $G$-action on $R$ is homologically small.
\item[(3)]
There is a natural isomorphism of graded algebras
$R\ast G\cong\End_{R^G}(R).$
\end{enumerate}
\end{theorem}

The Auslander theorem is one step in establishing the noncommutative
McKay correspondence. With this in place one should expect to extend
other parts of the McKay correspondence to the noncommutative world
\cite{CKWZ1, CKWZ2}.

This paper is organized as follows. We provide background material
in Section 1. In Section 2, we recall some results in \cite{BHZ}.
In Section 3, we prove Theorem \ref{xxthm0.2}. In Section 4, we
prove Theorem \ref{xxthm0.4} and Corollary \ref{xxcor0.5}.
Theorem \ref{xxthm0.8} is proven in Section 5 and Theorem \ref{xxthm0.6}
is proven in Section 6. In Section 7 we give some comments about
different definitions of smallness and argue that the homological
smallness is probably the best replacement of the classical
smallness when having Hopf algebra actions.

\subsection*{Acknowledgments}
The authors are very grateful  to Kenneth Chan, Andrew Conner,
Frank Moore, Fredy Van Oystaeyen, Chelsea Walton, Robert Won,
Yinhuo Zhang for sharing their (un-published) notes and interesting
ideas; specially to Jason Gaddis and Ellen Kirkman for pointing out
an error in an earlier version of the paper and for many useful
conversations on the subject. The authors thank the referee for 
his/her very careful reading and extremely valuable comments. 
The authors also thank
Van Nguyen for suggesting the word ``congenial'' in the
Definition \ref{xxdef4.3}.
Both Y.-H. Bao and J.-W. He were supported by NSFC
(Nos. 11571239, 11401001)
and J.J. Zhang by the US National Science
Foundation (grant No. DMS 1402863).

\section{Preliminaries}
\label{xxsec1}

Throughout let $\Bbbk$ be a base ring that is a noetherian commutative
domain. Unless otherwise stated, algebraic objects are over $\Bbbk$.
Let $A$ be a (left and right) noetherian algebra. Usually we are
working with {\it right} $A$-modules, and mostly, with finitely generated
(or f.g. for short) {\it right} $A$-modules. We write $\Mod A$ for
the category of all right $A$-modules, and $\mmod A$ for the full
subcategory of all  f.g. right $A$-modules.

In this paper we  mainly use $\GKdim$ or $\Kdim$. However,
it is a good idea to introduce some definitions for
an abstract dimension function. When $\Bbbk$ is not a field,
the GK-dimension is given as in \cite[Definition 1.1]{BHZ}
or  \cite[Lemma 3.1]{BZ}.

\begin{definition}
\label{xxdef1.1}
Let $\partial$ denote a function on f.g. right
$A$-modules
$$\partial: \mmod A\longrightarrow {\mathbb R}\cup\{\pm \infty\}.$$
\begin{enumerate}
\item[(1)]
We say $\partial$ is a {\it dimension function} if all f.g. $A$-modules $M$,
\begin{equation}
\label{E1.1.1}\tag{E1.1.1}
\partial(M)\geq \max\{\partial(N),\partial(M/N)\}
\end{equation}
whenever $N$ is a submodule of $M$.
\item[(2)]
We say $\partial$ is {\it exact} if for all f.g. $A$-modules $M$,
\begin{equation}
\label{E1.1.2}\tag{E1.1.2}
\partial(M)=\max\{\partial(N),\partial(M/N)\}
\end{equation}
whenever $N$ is a submodule of $M$.
\item[(3)]
Suppose $B$ is an overring of $A$ such that $B$ is noetherian
and $B_A$ is f.g. Assume that the dimension
function $\partial$ is also defined on right $B$-modules.
\begin{enumerate}
\item[(3a)]
We say $\partial$ is {\it weakly $B/A$-hereditary} if $\partial(M_A)
\leq \partial(M_B)$ for every f.g. right $B$-module $M$.
\item[(3b)]
We say $\partial$ is
{\it $B/A$-hereditary} if, for every f.g. right
$B$-module $M$,
\begin{equation}
\label{E1.1.3}\tag{E1.1.3}
\partial(M_A)=\partial(M_B).
\end{equation}
\end{enumerate}
\item[(4)]
Let $C$ be another noetherian algebra and suppose $\partial$ is also a
dimension function on {\it left} $C$-modules. We say $\partial$
is {\it $(C,A)$-symmetric} if, for
every $(C,A)$-bimodule $M$ that is f.g. over both sides, one has
\begin{equation}
\label{E1.1.4}\tag{E1.1.4}
\partial(_CM)=\partial(M_A).
\end{equation}
\item[(5)]
Let $D$ be a noetherian algebra and suppose that the dimension
function $\partial$ is also defined  on {\it right} $D$-modules.
We say $\partial$ is {\it $(A,D)_i$-torsitive} if, for every $(A,D)$-bimodule
$M$ f.g. on both sides and every f.g. right $A$-module $N$, one has
\begin{equation}
\label{E1.1.5}\tag{E1.1.5}
\partial(\Tor^A_j(N,M)_D)\leq \partial (N_A)
\end{equation}
for all $j\leq i$.
\end{enumerate}
\end{definition}

The definition of a dimension function given in
\cite[6.8.4]{MR} is stronger than the definition in
Definition \ref{xxdef1.1}(1).
The word ``torsitive'' stands for ``Tor transitive''.
We collect some facts about $\GKdim$ and $\Kdim$.
If $M$ is a f.g. graded right module over a
noetherian  locally finite graded algebra $B$,
then its GK-dimension can be computed by \cite[(E7)]{Zh1}
\begin{equation}
\label{E1.1.6}\tag{E1.1.6}
\GKdim M
=\underset{k\to \infty}{\overline{\lim}}\log_k\sum_{j\leq k}\dim (M_j).
\end{equation}

Unless otherwise stated, a graded algebra in this paper means
${\mathbb N}$-graded.
Let $B$ be a filtered algebra with an ${\mathbb N}$-filtration
${\mathcal F}$:
$$0\subseteq F_0 B\subseteq F_1B\subseteq\cdots \subseteq
F_nB\subseteq\cdots.$$
The {\it associated graded ring} is defined to be
$$\gr_{\mathcal F} B:=\bigoplus_{i=0}^{\infty} F_i B/F_{i-1} B.$$
where $F_{-1}B=0$.
We say ${\mathcal F}$ is {\it exhaustive} if $B=\bigcup_{i}
F_i B$. In this paper all filtrations are exhaustive.

\begin{lemma}
\label{xxlem1.2}
Let $A$ and $B$ be noetherian algebras.
\begin{enumerate}
\item[(1)]
If $A$ is an ${\mathbb N}$-filtered algebra such that the
associated graded algebra is locally finite and noetherian,
then $\GKdim$ is exact.
\item[(2)] If $B$ is an overring of $A$, then
$\GKdim$ is weakly $B/A$-hereditary.
\item[(3)]
Let $A$ be a subring of $B$.
Suppose that $B$ is ${\mathbb N}$-filtered such that
$\gr_{\mathcal F} B$ is noetherian and locally finite graded and
that $\gr_{\mathcal F} A$ induced by the filtration on $B$ is a
noetherian and locally finite graded subalgebra of $gr_{\mathcal F} B$.
If modules $B_A$ and $(\gr_{\mathcal F} B)_{\gr_{\mathcal F} A}$ are f.g., then
$\GKdim$ is $B/A$-hereditary.
\item[(4)]
$\GKdim$ is $(A, B)$-symmetric.
\item[(5)]
$\GKdim$ is $(A,B)_0$-torsitive.
\item[(6)]
\cite[Lemma 1.6]{BHZ}
Let $A$ and $B$ be noetherian and locally finite graded. Then
$\GKdim$ is $(A,B)_{\infty}$-torsitive in the graded setting.
\end{enumerate}
\end{lemma}

\begin{proof} (1) We can pass to the case when $\Bbbk$ is a field.
Then the assertion follows from \cite[Theorem 6.14]{KL}.

(2) This is a consequence of the definition of
$\GKdim$.

(3) By localization we pass to the case when $\Bbbk$ is a field.
Let $M$ be a f.g. right $B$-module. There is a finite
dimensional subspace $V\subseteq M$ such that $M=VA=VB$.
Let ${\mathcal F}_B:=\{F_i B\}_{i\geq 0}$ be the
filtration on $B$ and let
${\mathcal F}_A:=\{F_i A:= F_i B\cap A\}_{i\geq 0}$ be the
induced filtration on $A$. Define a filtration ${\mathcal F}_M$ on
$M$ by $F_i M:=V F_i B$ for all $i$. Then ${\mathcal F}_M$ is an
${\mathcal F}_B$-filtration and $\gr_{\mathcal F} M$ is a
f.g. right $\gr_{\mathcal F} B$-module. By \cite[Proposition 6.6]{KL},
$\GKdim M_B=\GKdim (\gr_{\mathcal F} M)_{\gr_{\mathcal F} B}$.
Note that ${\mathcal F}_M$ is also an
${\mathcal F}_A$-filtration and $\gr_{\mathcal F} M$ is a f.g.
right $\gr_{\mathcal F} A$-module.
Since $(\gr_{\mathcal F} M)_{\gr_{\mathcal F} B}=
(\gr_{\mathcal F} M)_{\gr_{\mathcal F} A}$ as graded vector
spaces, by \eqref{E1.1.6},
$$\GKdim (\gr_{\mathcal F} M)_{\gr_{\mathcal F} B}=
\GKdim (\gr_{\mathcal F} M)_{\gr_{\mathcal F} A}.$$
Combining these statements we have
$$\GKdim M_B=\GKdim (\gr_{\mathcal F} M)_{\gr_{\mathcal F} B}
=\GKdim (\gr_{\mathcal F} M)_{\gr_{\mathcal F} A}
=\GKdim M_A.$$

(4) By localization we pass to the case when $\Bbbk$ is a field.
The assertion follows from \cite[Lemma  5.3]{KL}

(5) This follows from \cite[Proposition 5.6]{KL} with a slight
modification.
\end{proof}

The next lemma concerns the Krull dimension, $\Kdim$.

\begin{lemma}
\label{xxlem1.3}
\begin{enumerate}
\item[(1)]
\cite[Lemma 6.2.4]{KL}
$\Kdim$ is exact.
\item[(2)]
\cite[Proposition 6.4.13]{MR}
Suppose $A$ and $C$ are noetherian and PI. Then
$\Kdim$ is $(A, C)$-symmetric.
\end{enumerate}
\end{lemma}

We recall some more definitions. The first one is  from \cite{BHZ},
which is similar to torsitivity (Definition \ref{xxdef1.1}(5)) 
of a dimension function.

\begin{definition}
\cite[Definition 1.2]{BHZ}
\label{xxdef1.4}
Let $A$ and $B$ be noetherian algebras and
$\partial$ be an exact dimension function that is defined on right
$A$-modules and on right $B$-modules. Let $n$ and $i$ be nonnegative
integers. Let $_AM_B$ denote any bimodule which is f.g. both as a
left $A$-module and as a right $B$-module.
\begin{enumerate}
\item[(1)]
We say $\partial$ {\it satisfies $\gamma_{n,i}(M)$} if, for every f.g.
right $A$-module $N$ with $\partial(N_A)\leq n$, one has
$\partial(\Tor^A_j(N, M)_B)\leq n$
for all $0\leq j\leq i$.
\item[(2)]
We say $\partial$ {\it satisfies $\gamma_{n,i}$} if it satisfies
$\gamma_{n,i}(M)$ for all $M$ given as above.
\end{enumerate}
\end{definition}

\begin{definition}\cite[Definition 0.4]{ASZ1}
\label{xxdef1.5}
Let $B$ be an algebra and $M$ a right $B$-module.
\begin{enumerate}
\item[(1)]
Let $\partial$ be a dimension function. We say $B$ is
{\it $\partial$-Cohen-Macaulay} (or {\it $\partial$-CM}) if
$\partial (B)=d\in\mathbb{N}$,
and
$$j(M)+\partial (M)=\partial (B),$$
for every f.g. right $B$-module $M\neq 0$.
\item[(2)]
If $B$ is $\GKdim$-Cohen-Macaulay, we just
say it is {\it Cohen-Macaulay} or {\it CM}.
\end{enumerate}
\end{definition}

The CM property together with the Artin-Schelter regularity
and the Auslander property (another homological property)
has been studied in the noncommutative setting in
\cite{ASZ1, ASZ2, SZ, YZ, Zh1}.

\section{Hypotheses and results from \cite{BHZ}}
\label{xxsec2}

We repeat the following hypotheses given in \cite[Section 2]{BHZ}.

\begin{hypothesis} \cite[Hypothesis 2.1]{BHZ}
\label{xxhyp2.1}
\begin{enumerate}
\item[(1)]
$A$ and $B$ are noetherian algebras.
\item[(2)]
Let $e$ be an idempotent in $B$ and $A=eBe$.
\item[(3)]
$\partial$ is a dimension function defined on right $A$-modules and
right $B$-modules and $2\leq d:=\partial(B)< \infty$.
\item[(4)]
$B$ is a $\partial$-CM algebra.
\item[(5)]
The left $A$-module $eB$ and the right $A$-module $Be$ are f.g..
\item[(6)]
$\partial$ is exact on right $A$-modules and right $B$-modules.
\item[(7)]
For every f.g. right $B$-module $N$,
$\partial ((Ne)_A)\leq \partial (N_B)$.
\end{enumerate}
\end{hypothesis}

The original Hypothesis \ref{xxhyp2.1}(4) in \cite{BHZ} is weaker.
We use the present form to avoid another technical definition.
Under Hypothesis \ref{xxhyp2.1}(2), $Be$ is a right
$A$-module and there is a natural algebra morphism
\begin{equation}
\label{E2.1.1}\tag{E2.1.1}
\varphi: \; B\to \End_{A}(Be),\quad \varphi(b)(b'e)=bb'e
\end{equation}
induced by the left multiplication. Here is one of the main
results in \cite{BHZ}.

Some parts of hypotheses are satisfied automatically.

\begin{lemma}
\label{xxlem2.2}
Under Hypothesis {\rm{\ref{xxhyp2.1}(1)-(6)}},
Hypothesis {\rm{\ref{xxhyp2.1}(7)}} holds if one of the following is true.
\begin{enumerate}
\item[(i)]
$\partial$ is $(B,A)_0$-torsitive.
\item[(ii)]
$\partial=\GKdim$.
\item[(iii)]
$B$ is PI and $\partial=\Kdim$.
\end{enumerate}
\end{lemma}

\begin{proof}
(i) Hypothesis {\rm{\ref{xxhyp2.1}(7)}} can be written as
$\partial(( N\otimes_B Be)_A)\leq \partial (N_B)$ which is
a special case of $(B,A)_0$-torsitivity.

(ii) This follows by case (i) and Lemma \ref{xxlem1.2}(5).

(iii) By Lemma \ref{xxlem3.1}(2) in the next section, $\partial$ is
$(B,A)_{\infty}$-torsitive. The assertion follows by case (i).
\end{proof}

\begin{theorem}
\label{xxthm2.3}\cite[Theorem 2.4]{BHZ}
Let $(A,B)$ satisfy Hypothesis {\rm{\ref{xxhyp2.1}(1-7)}}. Suppose
\begin{equation}
\label{E2.3.1}\tag{E2.3.1}
{\text{$\partial$ satisfies $\gamma_{d-2,1}(eB)$.}}
\end{equation}
Then the following statements are equivalent.
\begin{enumerate}
\item[(i)]
The functor $-\otimes_{\mathcal{B}}\mathcal{B}e: \;
\qmod_{d-2} B\longrightarrow \qmod_{d-2} A$ is an equivalence.
\item [(ii)]
The natural map $\varphi$ of \eqref{E2.1.1} is an isomorphism
of algebras.
\item [(iii)] $\partial (B/(BeB))\leq d-2$.
\end{enumerate}
\end{theorem}

We mainly concern with parts (ii) and (iii), so will not explain
part (i) in the above theorem (definitions can be found in \cite{BHZ}).
There is also a graded version
of Theorem \ref{xxthm2.3}. For applications to the Auslander
theorem, we need have a Hopf action.

Let $(H,\Delta,\varepsilon, S)$ be a Hopf algebra that is free
of finite rank over the base commutative domain $\Bbbk$.
Let $R$ be a noetherian algebra and
assume that $R$ is a left $H$-module algebra, i.e., $H$ acts on
$R$. Then the smash product $R\#H$ is a noetherian algebra.
Suppose that $H$ has a (left and right) integral $\inth$ such that
$\varepsilon(\inth)=1$. If $\Bbbk$ is a field, this is equivalent to
the fact that $H$ is semisimple. Let $e=1\# \inth\in R\#H$. One sees
that $e$ is an idempotent of $R\#H$. The fixed subring of the
$H$-action on $R$ is
\begin{equation}
\notag
R^H:=\{r\in R|h\cdot r=\varepsilon(h)r,\forall\ h\in H\}.
\end{equation}
Then $R^H$ is a subalgebra of $R$. We recall the hypothesis
in \cite[Section 3]{BHZ} in the ungraded setting.

\begin{hypothesis}
\cite[Hypothesis 3.2]{BHZ}
\label{xxhyp2.4}
\begin{enumerate}
\item[(1)]
$R$ is a noetherian  algebra.
\item[(2)]
$H$ is a Hopf algebra given as above with an action on $R$.
\item[(3)]
Let $B$ be the algebra $R\# H$ with $e:=1\#\inth\in B$.
Identifying $R^H$ with $eBe$ by \cite[Lemma 3.1(3)]{BHZ},
$R$ with $Be$ by \cite[Lemma 3.1(5)]{BHZ}.
\item[(4)]
Let $\partial$ be an exact dimension function on right f.g. $B$-modules,
$R$-modules, and $R^H$-modules, and $\partial(R)=:d\geq 2$.
\item[(5)]
$\partial$ is $B/R$-hereditary.
\item[(6)]
$R$ is $\partial$-CM.
\end{enumerate}
\end{hypothesis}

Again Hypothesis \ref{xxhyp2.4}(6)  is stronger than the original version in
\cite[Hypothesis 3.2(6)]{BHZ} to avoid a new concept. On the other hand,
the original version requires $R$ to be locally finite
graded, which is removed here. We need an ungraded
version of \cite[Proposition 3.3]{BHZ}.

\begin{proposition} \cite[Proposition 3.3]{BHZ}
\label{xxpro2.5}
Retain Hypothesis {\rm{\ref{xxhyp2.4}(1-6)}}. Assume
that $\Bbbk$ is a field. Then $B:=R\#H$ is $\partial$-CM.
\end{proposition}

\begin{proof} Repeat the proof of \cite[Proposition 3.3]{BHZ}
in the ungraded setting.
\end{proof}

\begin{lemma}
\label{xxlem2.6} Suppose $\partial$ is either $\GKdim$ when
algebras are in Lemma {\rm{\ref{xxlem1.2}(1)}} or $\Kdim$
when algebras are PI. Assume that $\Bbbk$ is a field.
Under Hypothesis {\rm{\ref{xxhyp2.4}}},
Hypothesis {\rm{\ref{xxhyp2.1}}} holds.
\end{lemma}

\begin{proof} We need to check (1)-(7) in Hypothesis {\rm{\ref{xxhyp2.1}}}.

(1) Since $R$ is noetherian, so is $B:=R\# H$.
By \cite[Lemma 3.1(1)]{BHZ}, $A:=R^H$ is noetherian.

(2) Follows from Hypothesis \ref{xxhyp2.4}(3).

(3) Since $\partial$ is either $\GKdim$ or $\Kdim$, this is not hard
to check.

(4) This is Proposition \ref{xxpro2.5}.

(5) This follows from \cite[Lemma 3.1(2)]{BHZ}.

(6) When $\partial=\GKdim$, this is Lemma \ref{xxlem1.2}(1)
and when $\partial=\Kdim$, this is Lemma \ref{xxlem1.3}(1).

(7) This is Lemma \ref{xxlem2.2}.
\end{proof}

We conclude this section by stating a result of \cite{BHZ}.

\begin{theorem}\cite[Theorem 3.5]{BHZ}
\label{xxthm2.7}
Retain Hypothesis {\rm{\ref{xxhyp2.4}(1-3)}} and assume that
$R$ is locally finite and graded over a base field
$\Bbbk$. Further assume that $R$ is CM with $\GKdim R\geq 2$.
Set $(A,B)=(R^H,R\#H)$. Then the following are equivalent.
\begin{enumerate}
\item[(i)]
The natural map
$\varphi:B\longrightarrow\text{\rm End}_A(R)$ is an
isomorphism of algebras.
\item[(ii)]
$\Pty(R,H)\geq 2$.
\item[(iii)]
The $H$-action on $R$ is  h.small.
\end{enumerate}
\end{theorem}

\section{The PI case}
\label{xxsec3}

{\it In this section the base ring $\Bbbk$ is a field.}
We start with a few lemmas.

\begin{lemma}
\label{xxlem3.1}
All algebras mentioned in this lemma are noetherian and PI.
Let $B$ be an algebra and let $A$ be a subalgebra of $B$ such that
$B_A$ is f.g.. Let $\partial$ be a dimension function in the sense
of \cite[\S6.8.4, p.224]{MR} defined on left and right $A$- and $B$-modules
such that
\begin{enumerate}
\item[(i)]
$\partial$ is exact and symmetric over any two algebras.
\item[(ii)]
$\partial$ is an integer-valued function.
\item[(iii)]
$\partial(M)\geq 0$ for all $M\neq 0$ and $\partial(0)=-1$.
\end{enumerate}
Then the following hold.
\begin{enumerate}
\item[(1)]
Assume that $\partial (B)<\infty$. Then
$\partial$ is $B/A$-hereditary.
\item[(2)]
If $C$ is another algebra and $\partial$ is defined on right $C$-modules, 
then $\partial$ is $(C,B)_{\infty}$-torsitive.
\end{enumerate}
\end{lemma}

\begin{proof} (1) We show that, if $M$ is
a f.g. right $B$-module, then
$$\partial (M_A)= \partial (M_B),$$
by induction on $\partial (M_B)$. Nothing needs to be proven if
$\partial (M_B)<0$. Suppose that
$\partial (N_A)=\partial (N_B)$ for all f.g. right
$B$-modules $N$ with $\partial (N_B)< n$. Let $M$ be a
f.g. right $B$-module with $\partial (M_B)=n$.
Since $\partial$ is exact and $B$ is PI, we can assume that $M_B$
is critical and $M_B$ is a submodule of $B/{\mathfrak p}$ for
some prime ideal ${\mathfrak p}\subseteq B$ (see \cite[Lemma 2.1]{SZ}).
Replacing $M_B$ by $M_B^{\oplus s}$ for some $s$, we can
assume that $M_B$ is essential in $T:=B/{\mathfrak p}$. Then
$\partial ((T/M)_B)< n$.
So $\partial (M_B)=\partial (T_B)$ by the exactness. By
induction hypothesis, $\partial ((T/M)_A)=\partial ((T/M)_B)< n$.

Note that $T$ is a $B$-bimodule and a $(B,A)$-bimodule,
which is f.g. over $A$ and $B$ on both sides. By the symmetry
of $\partial$,
$$\partial (T_A)=\partial ({_BT})=\partial (T_B)=\partial (M_B)=n.$$
Since $\partial (T_A)= n$ and $\partial ((T/M)_A)< n$, by the exactness,
$$\partial (M_A)=\partial (T_A)=n=\partial (M_B)$$
which finishes the induction.

(2) Fix a $(C,B)$-bimodule $S$ that is f.g. on both sides.
We show that, if $M$ is a f.g. right $C$-module, then
$$\partial (\Tor^C_i(M,S)_B)\leq \partial (M_C),$$
by induction on $\partial (M_C)$. Nothing needs to be proven if
$\partial (M)<0$.
Suppose that the assertion holds for all $M$ that
is a f.g. right $C$-module with $\partial (M)<n$.
Let $M$ be a f.g. right $C$-module with $\partial (M)=n$.
Note that $M$ has a $\Kdim$-critical composition series 
$$M=M_0\supseteq M_1\supseteq M_2\supseteq \cdots\supseteq M_n=0,$$ 
such that $M_j/M_{j+1}$ is $\Kdim$-critical and 
$\partial(M_j/M_{j+1})\leq n$ for all $j=0,\dots,n-1$. 
If we can show $\partial(\Tor_i^C(M_j/M_{j+1},S)_B)\leq 
\partial(M_j/M_{j+1})$ for all $j$, then it follows that 
$\partial(\Tor_i^C(M,S)_B)\leq n$ by the exactness of $\partial$ 
and the long exact sequence of $\Tor_i^C$. Hence we may assume that 
$M$ itself is $\Kdim$-critical. Now by \cite[Lemma 2.1]{SZ}(ii), 
$M_C$ is a submodule of $C/{\mathfrak p}$ for
some prime ideal ${\mathfrak p}\subseteq C$. Replacing $M$ by
$M^{\oplus s}$ for some $s$, we can assume that $M_C$ is
essential in $T:=C/{\mathfrak p}$. Then $\partial (T/M)<n$.
So $\partial (M)=\partial (T)=n$ by the exactness. By
induction hypothesis, for all $i$,
$$\partial (\Tor^C_i((T/M), S))\leq \partial (T/M)< n.$$
By using the long exact sequence of $\Tor^C_i$ again,
$$\partial (\Tor^C_i(M,S))\leq \max\{n-1, \partial (\Tor^C_i(T, S))\}_{i\geq 0}.$$
It remains to show that $\partial (\Tor^C_i(T, S))\leq n$
for all $i$. Since $C$ is noetherian, $T$ and $S$ are f.g. on both sides,
the $(C,B)$-bimodule $W:=\Tor^C_i(T,S)$ is f.g. on both sides.
By the symmetry of $\partial$,
$$\partial (W_B)=\partial (_CW)\leq \partial (_CT)=\partial (T_C)=n.$$
as required.
\end{proof}

For the rest of this section we take $\partial=\Kdim$. Following
\cite{SZ}, an (${\mathbb N}$-graded) algebra $A$ is called
{\it {\rm{(}}graded{\rm{)}} injectively smooth} if $\injdim A=n<\infty$ and
$\Ext^n_A(S,A)\neq 0$ for all simple right (graded) $A$-modules $S$.
We will use the following result.

\begin{lemma}
\label{xxlem3.2}
Let $A$ be noetherian and PI. The following hold.
\begin{enumerate}
\item[(1)]
\cite[Theorem 1.3]{SZ}
Let $A$ be a {\rm{(}}graded{\rm{)}} injectively smooth algebra. Then $A$ is $\Kdim$-CM.
\item[(2)]
\cite[Corollary 3.13]{SZ}
If $A$ is local of finite injective dimension, then $A$ is $\Kdim$-CM.
Consequently, $A$ is injectively smooth.
\item[(3)]
\cite[Theorem 1.1]{SZ}
If $A$ is connected graded of finite injective dimension, then $A$ is CM.
Consequently, $A$ is injectively smooth.
\end{enumerate}
\end{lemma}

We have the following version of the Auslander theorem.

\begin{theorem}
\label{xxthm3.3}
Let $R$ be a noetherian PI and $\Kdim$-CM algebra of $\Kdim \geq 2$
and $H$ a semisimple Hopf algebra acting on $R$.
Then the following are equivalent.
\begin{enumerate}
\item[(1)]
There is a natural isomorphism of algebras $R\# H\cong \End_{R^H}(R)$.
\item[(2)]
$\Pty_{\Kdim}(R,H)\geq 2$.
\item[(3)]
The $H$-action on $R$ is  h.small.
\end{enumerate}
\end{theorem}

\begin{proof} First we check Hypothesis \ref{xxhyp2.4}(1-6).
(1)-(4) are clear by taking $\partial=\Kdim$. Item (5)
follows from Lemmas \ref{xxlem1.3}(1) and
\ref{xxlem3.1}(1). (6) is a hypothesis. By Lemma
\ref{xxlem2.6}, Hypothesis \ref{xxhyp2.1} holds.
By Lemma \ref{xxlem3.1}(2), $\Kdim$ is
$(A,B)_{\infty}$-torsitive. Therefore \eqref{E2.3.1}
holds. By Theorem \ref{xxthm2.3}, (1) is equivalent to
(2). Since $R$ is $\Kdim$-CM, (2) is equivalent to (3).
\end{proof}

Now we are ready to prove Theorem \ref{xxthm0.2}.
Recall that $\Kdim M=\GKdim M$ if $M$ is a f.g. right module
over an affine noetherian PI algebra \cite[Lemma 4.3(i)]{SZ}.

\begin{proof}[Proof of Theorem \ref{xxthm0.2}]
When $R$ is PI local or connected graded of finite
injective dimension, $R$ is $\Kdim$-CM by Lemma \ref{xxlem3.2}.
The assertion follows by Theorem \ref{xxthm3.3}.
\end{proof}

We consider some explicit examples. The next is an immediate
consequence of Theorem \ref{xxthm3.3}.

The skew polynomial ring $\Bbbk_{p_{ij}}[x_1,\cdots,x_n]$ is defined
in Example \ref{xxex5.2}.

\begin{corollary}
\label{xxcor3.4}
Let $p_{ij}$ be roots of unity for all $1\leq i<j\leq n$ and
let $R$ be the skew polynomial ring $\Bbbk_{p_{ij}}[x_1,\cdots,x_n]$.
Suppose $H$ is a semisimple Hopf algebra acting on $R$. Then
the following are equivalent.
\begin{enumerate}
\item[(1)]
There is a natural isomorphism of algebras $R\# H\cong \End_{R^H}(R)$.
\item[(2)]
$\Pty (R,H)\geq 2$.
\item[(3)]
The $H$-action on $R$ is  h.small.
\end{enumerate}
\end{corollary}

\begin{proof} It is well-known that $R$ is CM. The assertion
follows from Theorem \ref{xxthm3.3} and the fact $\Kdim=\GKdim$.
\end{proof}

To study other examples, we need to consider filtered algebras.
Let $B$ be a filtered algebra with an ${\mathbb N}$-filtration
${\mathcal F}$:
$$0\subseteq F_0 B\subseteq F_1B\subseteq\cdots \subseteq
F_nB\subseteq\cdots.$$
For an element $e\in F_0B$, let $\overline{B}$ be the
factor ring $B/(e)$ where $(e)$ is the ideal of $B$
generated by $e$. Let $\pi:B\to \overline{B}$ be the
canonical projection map. Then $\overline{B}$ is also
a filtered algebra with the filtration $\overline{\mathcal F}$
induced from ${\mathcal F}$. Let $\gr e$ be the element in
$\gr_{\mathcal F} B$ corresponding to $e$, which has degree
0 as $e\in F_0B$. The following observation is obvious.

\begin{lemma}
\label{xxlem3.5}
The projection map induces an epimorphism of graded algebras
$$(\gr_{\mathcal F} B)/(e')\to \gr_{\overline{\mathcal F}} \overline{B},$$
where $e':=\gr e$ is viewed as an element of degree $0$ in $\gr_{\mathcal F} B$.
\end{lemma}

\begin{proof} First of all, we have a surjective homomorphism
$\phi: \gr_{\mathcal F} B\to \gr_{\overline{\mathcal F}} \overline{B}$ induced
by the surjection
$$F_i B/F_{i-1} B\to (F_i B+(e))/(F_{i-1}+(e))$$
for all $i$. It is clear that $\phi$ maps $e'$ to 0.
The assertion follows.
\end{proof}

\begin{proposition}
\label{xxpro3.6}
Let $R$ be a filtered algebra with an ${\mathbb N}$-filtration
${\mathcal F}$ such that $\gr_{\mathcal F} R$ is locally finite
and noetherian. Assume that both $R$ and $\gr_{\mathcal F} R$
are CM. Let $H$ be a semisimple Hopf algebra acting on
$R$. Suppose that the $H$-action
on $R$ preserves the filtration ${\mathcal F}$. Then
$$\Pty(R,H)\geq \Pty(\gr_{\mathcal F} R, H).$$
As a consequence, if the Auslander theorem holds for
the $H$-action on $\gr_{\mathcal F} R$ and if
$\gamma_{d-2,1}({_{R^H}R_{R\# H}})$ holds,
where $d=\GKdim R$,
then the Auslander theorem holds for the $H$-action on $R$.
\end{proposition}

\begin{proof}
First of all, $H$ acts on $\gr_{\mathcal F} R$ naturally.
Since $\gr_{\mathcal F} R$ is locally finite and noetherian,
$\GKdim \gr_{\mathcal F} R=\GKdim R$ by \cite[Proposition 6.6]{KL}.

Let $e:=1\# \inth\in R\# H$ where $\inth$ is the integral of $H$ and let
$e':=1\# \inth\in (\gr_{\mathcal F} R)\# H$. We define a filtration
${\mathcal F}'$ on $B:=R\# H$ by
$$F'_i B= (F_i R)\# H$$
for all $i\geq 0$. Then the associated graded ring of $B$ is
$$\gr_{{\mathcal F}'} B=\gr_{{\mathcal F}'} (R\#H)
\cong (\gr_{\mathcal F} R)\# H.$$
Let $\overline{B}=B/(e)$ and $\overline{\mathcal F}$ be the
filtration on $\overline{B}$ induced by the filtration ${\mathcal F}'$.
Now we have
$$\begin{aligned}
\GKdim ((\gr_{\mathcal F} R)\# H)/(e'))&=\GKdim
\left((\gr_{{\mathcal F}'} B)/(e')\right)\\
&\geq \GKdim \gr_{\overline{\mathcal F}} (B/(e))  
 \qquad {\text{by Lemma \ref{xxlem3.5}}}\\
&=\GKdim \gr_{\overline{\mathcal F}} (\overline{B})\\
&=\GKdim \overline{B}
\end{aligned}
$$
where the last equation follows by \cite[Proposition 6.6]{KL}. Therefore
$$\begin{aligned}
\Pty(R,H)&=\GKdim R-\GKdim (R\# H)/(e)=\GKdim R-\GKdim \overline{B}\\
&=\GKdim \gr_{\mathcal F} R-\GKdim \overline{B}\\
&\geq \GKdim \gr_{\mathcal F} B-\GKdim ((\gr_{\mathcal F} R)\# H)/(e'))\\
&=\Pty(\gr_{\mathcal F} R,H).
\end{aligned}
$$
The consequence follows from Theorems
\ref{xxthm2.3} and \ref{xxthm2.7}, and Lemma \ref{xxlem2.6}.
\end{proof}

In the next example, $R$ is neither local nor connected graded.

\begin{corollary}
\label{xxcor3.7}
Suppose ${\rm{char}}\; \Bbbk \nmid 2n$.
Let $R$ be the quantum Weyl algebra
generated by $x_1,\cdots,x_n$ and
subject to the relations $x_ix_j+x_jx_i=1$
for all $i\neq j$. Let $G$ be the group
generated by $\sigma: x_i\to x_{i+1}$
for all $i<n$ and $x_n\to x_1$. Then
$\Pty(R,G)\geq 2$ and $R\ast G\cong
\End_{R^G}(R)$.
\end{corollary}

\begin{proof}
Using the standard filtration defined by
$F_i R=(\Bbbk+\sum_{s=1}^n \Bbbk x_s)^i$,
we have  $\gr_{\mathcal F} R\cong \Bbbk_{-1}
[x_1,\cdots,x_n]$. By \cite[Theorem 0.5]{BHZ},
$\Pty(\gr_{\mathcal F} R, G)\geq 2$.
By Proposition \ref{xxpro3.6}, $\Pty(R,G)\geq
\Pty(\gr_{\mathcal F}R,G)\geq 2$.
Since $R$ is affine PI and noetherian,
$\GKdim M=\Kdim M$ for every f.g. right $R$-module.
It is well-known that both $R$ and $\gr_{\mathcal F} R$
are CM. By Theorem \ref{xxthm3.3}, $R\ast G\cong
\End_{R^G}(R)$.
\end{proof}

\section{Reduction mod-$p$ and universal enveloping algebras}
\label{xxsec4}

The goal of this section is to prove Theorem \ref{xxthm0.4} and
Corollary \ref{xxcor0.5}, which requires quite a bit of preparation.
In this section we assume that $\Bbbk$ is a field
with prime subring $\Bbbk_0={\mathbb Z}$ or ${\mathbb F}_{p}:
={\mathbb Z}/(p)$,  for some prime $p$, inside $\Bbbk$.
Let $D$ denote any $\Bbbk_0$-affine subalgebra
of $\Bbbk$. Such a $D$ is an {\it admissible domain} in the sense of
\cite[p.580]{ArSZ}. Note that if $D$ is a $\Bbbk_0$-affine
subalgebra of $\Bbbk$, so is $D[s^{-1}]$ for any nonzero
element $s\in D$.

\begin{definition}
\label{xxdef4.1} Let $A$ be an algebra over $\Bbbk$ and $M$ a
right $A$-module. Let $D$ be a $\Bbbk_0$-affine subalgebra of $\Bbbk$.
\begin{enumerate}
\item[(1)]
A $D$-subalgebra $A_D$ of $A$ is called an {\it order} of
$A$ if the following hold
\begin{enumerate}
\item[(1a)]
$A_D$ is free over $D$,
\item[(1b)]
$A_D\otimes_D \Bbbk=A$.
\end{enumerate}
\item[(2)]
An $A_D$-submodule $M_D$ of $M$ is called an {\it order} of
$M$ if the following hold
\begin{enumerate}
\item[(2a)]
$M_D$ is free over $D$,
\item[(2b)]
$M_D\otimes_D \Bbbk(=M_D\otimes_{A_D} A)=M$.
\end{enumerate}
\end{enumerate}
\end{definition}

The following lemma is easy. If $D\to F$ is an algebra
homomorphism of commutative rings, and if $A_D$ is an
algebra over $D$, define $A_F$ to be $A_D\otimes_D F$.
The module $M_F$ is defined similarly.

\begin{lemma}
\label{xxlem4.2}
Let $A_D$ and $B_D$ be two $D$-algebras.
Let $N$ be a right $A_D$-module and $M$ be an $(A_D,B_D)$-bimodule
which is $D$-central. Let $j\geq 0$ be an integer.
\begin{enumerate}
\item[(1)]
Assume that $F$ is flat over $D$ {\rm{(}}e.g., $F$ is a localization of $D${\rm{)}}. Then
$$\Tor_i^{A_D}(N,M)\otimes_D F=\Tor_i^{A_F}(N_F, M_F)$$
for all $i$.
\item[(2)]
Suppose $N, M$ and $\Tor_i^{A_D}(N,M)$ are free over $D$ for all
$i\leq j$.
Then, for every algebra homomorphism $D\to F$ of commutative rings,
we have
$$\Tor_i^{A_D}(N,M)\otimes_D F=\Tor_i^{A_F}(N_F, M_F)$$
for all $i\leq j$.
\end{enumerate}
\end{lemma}

\begin{proof}
(1) Since $M$ is $D$-central, $F\otimes_D M=M\otimes_D F=: M_F$.
Since $F$ is flat over $D$, the functor $-\otimes_D F$ is exact.
Then, for each $i\geq 0$,
$$\begin{aligned}
\Tor_i^{A_D}(N,M)\otimes_D F&= \Tor_i^{A_D}(N,M\otimes_D F)=\Tor_i^{A_D}(N,M_F)\\
&=\Tor_i^{A_D}(N,(A_F \otimes_{A_F} M_F))=\Tor_i^{A_F}((N\otimes_{A_D}A_F), M_F)\\
&=\Tor_i^{A_F}(N_F, M_F).
\end{aligned}
$$

(2) In part (2) we do not assume that $F$ is flat over $D$.
By the standard  Tor spectral sequence \cite[Theorem 11.51]{Ro},
$$\Tor_i^{A_D}(N,M)\otimes_D F= \Tor_i^{A_D}(N,M\otimes_D F)=\Tor_i^{A_D}(N,M_F)$$
since $M$ and $\Tor_i^{A_D}(N,M)$ are free over $D$ for all
$i\leq j$. Since $N$ is free over $D$, $\Tor_s^{A_D}(N, A_F)=0$
for all $s>0$. Then, by \cite[Theorem 11.51]{Ro}, we have
$$\begin{aligned}
\Tor_i^{A_D}(N,M_F)&=\Tor_i^{A_D}(N,(A_F\otimes_{A_F} M_F))
=\Tor_i^{A_F}((N\otimes_{A_D} A_F), M_F)\\
&=\Tor_i^{A_F}(N_F, M_F).
\end{aligned}
$$
Therefore the assertion follows.
\end{proof}

We recall some definitions for $D$-algebras. As  convention in
this paper,  a {\it filtration} of an algebra $A$ is an exhaustive
${\mathbb N}$-filtration. A {\it filtered} algebra is an algebra with
a filtration. Following \cite[p. 580]{ArSZ}, a $D$-algebra $A$
is called {\it strongly noetherian} if for every noetherian
commutative $D$-algebra $F$, $A\otimes_{D} F$ is noetherian. A
{\it graded} algebra means an ${\mathbb N}$-graded algebra and
{\it locally finite} means that each homogeneous piece is of
finite rank over $D$. We say a filtered algebra $A$ with
filtration ${\mathcal F}$ is {\it locally finite} if its associated graded
ring $\gr_{\mathcal F} A$ is locally finite.

\begin{definition}
\label{xxdef4.3}
A $\Bbbk$-algebra $A$ is called {\it congenial}
if the following hold.
\begin{enumerate}
\item[(1)]
$A$ is a noetherian locally finite filtered algebra
with an ${\mathbb N}$-filtration ${\mathcal F}$.
\item[(2)]
$A$ has an order $A_D$ where $D$ is a $\Bbbk_0$-affine
subalgebra of $\Bbbk$ such that $A_D$ is a noetherian
locally finite filtered algebra over $D$
with the induced filtration, still denoted by ${\mathcal F}$.
\item[(3)]
The associated graded ring $\gr_{\mathcal F} A_D$ is an
order of $\gr_{\mathcal F} A$.
\item[(4)]
$\gr_{\mathcal F} A_D$ is a strongly noetherian locally finite graded algebra over $D$.
\item[(5)]
If $F$ is a factor ring of $D$ and  is a finite field, then
$A_D\otimes_D F$ is an affine noetherian PI algebra over $F$.
\end{enumerate}
\end{definition}

Let $M$ be a f.g. right $A_D$-module. We say $M$ is
{\it generically free} \cite[p. 580]{ArSZ} if, there is a
simple localization $D[s^{-1}]$ for some $0\neq s\in D$,
such that $M[s^{-1}]$ is free over $D[s^{-1}]$. We will use the following
generic freeness result in \cite{ArSZ}.

\begin{lemma}
\label{xxlem4.4} Let $A$ be congenial. The following hold.
\begin{enumerate}
\item[(1)]
$A_D$ is a strongly noetherian locally finite filtered algebra over $D$.
\item[(2)]
Every f.g. right $A_D$-module $M$ is generically free. This
implies that if $D$ is replaced by $D[s^{-1}]$ for some $0\neq s\in D$,
$M$ becomes free over $D$.
\item[(3)]
$\gr_{\mathcal F} A$ is congenial.
\end{enumerate}
\end{lemma}

\begin{proof} (1) This follows from \cite[Proposition 4.10]{ArSZ} and
the fact that $\gr_{\mathcal F} A_D$ is a strongly noetherian locally finite graded algebra over $D$.

(2) This follows by part (1) and \cite[Theorem 0.3]{ArSZ}.

(3) This is clear.
\end{proof}

We will use generic freeness to prove some properties concerning
the GK-dimension of modules over a congenial algebra.

\begin{definition}
\label{xxdef4.5} Let $A$ be an algebra over $\Bbbk$ and let $M$ be
a right $A$-module. Let $D$ be a $\Bbbk_0$-affine subalgebra
of $\Bbbk$ and $A_D$ be an order of $A$.
We say that $M$ is {\it $\GKdim$-stable over $D$} if $M$ has
a $D$-order $M_D$ such that, for every algebra map $D\to F$ to a
noetherian commutative domain $F$,
\begin{equation}
\label{E4.5.1}\tag{E4.5.1}
\GKdim (M_D\otimes_D F)_{A_D\otimes_D F}=\GKdim M.
\end{equation}
\end{definition}

\begin{lemma}
\label{xxlem4.6}
If $M$ is $\GKdim$-stable over $D$, then $M_{D'}$ is $\GKdim$-stable over
$D'$ for every $\Bbbk_0$-affine subalgebra $D'$
of $\Bbbk$  containing $D$.
\end{lemma}

\begin{proof}
Let $f: D'\to F$ be an algebra map. Then $D\to D'\to F$ is an algebra
map. Applying \eqref{E4.5.1},  we have
$$
\GKdim (M_{D'}\otimes_{D'} F)_{A_{D'}\otimes_{D'} F}
=\GKdim (M_D\otimes_D F)_{A_D\otimes_D F}
=\GKdim M.
$$
The assertion follows.
\end{proof}

\begin{proposition}
\label{xxpro4.7}
Suppose that $A$ is congenial.
Let $M$ be a f.g. right $A$-module. Then there is an affine
$\Bbbk_0$-subalgebra $D\subseteq \Bbbk$ such that $M$ is
$\GKdim$-stable over $D$.
\end{proposition}

\begin{proof} By definition, there is a $D\subseteq \Bbbk$ such that
$A_D$ is an order of $A$. Since $A$ is noetherian and
$M$ is f.g, there is an exact sequence of right $A$-modules
$$A^{\oplus m}\xrightarrow{\phi} A^{\oplus n} \to M\to 0$$
for some integers $m$ and $n$.
Using a $D$-basis of $A_D$ as a $\Bbbk$-basis of $A$,
all coefficients of $\phi$ are in the subring $D'=D[c_i]$
for finitely many $c_i\in \Bbbk$. Replacing $D$ by $D'$ we might
assume that all coefficients of $\phi$ are in $D$. Let
$M_D$ be defined by the exact sequence
$$A_D^{\oplus m}\xrightarrow{\phi_D} A_D^{\oplus n} \to M_D\to 0,$$
where the matrix corresponding to $\phi_D$ is the same as the matrix
corresponding to $\phi$.
By Lemma \ref{xxlem4.4}(2), after replacing $D$ by $D[s^{-1}]$,
one can assume that $M_D$ is free over $D$. Since $\phi=\phi_{D}\otimes_D \Bbbk$,
$M_D\otimes_D \Bbbk=M$. Since $D\to \Bbbk$ is injective (and flat), we can view
$M_D$ as an $A_D$-submodule of $M$. Hence, $M_D$ is an order
of $M$.

Since $A$ has a filtration $\mathcal{F}$, we have the induced filtration, still
denoted by $\mathcal{F}$, of $A_D$. Let $M_D$ be a f.g. right $A_D$-module
with a finite set of generators $G$. Define a filtration on $M_D$ by
$F_i M_D= G (F_i A_D)$ for all $i\geq 0$. The associated graded module is
$$\gr_{\mathcal F} M_D:= \bigoplus_{i=0}^{\infty} F_i M_D/F_{i-1} M_D.$$
Then the graded right module $\gr_{\mathcal F} M_D$ is f.g.  over $\gr_{\mathcal F} A_D$.
In Definition \ref{xxdef4.3}, one can freely replace $D$ by a finitely generated
extension of $D$. Since $\gr_{\mathcal F} A_D$ is strongly noetherian and locally
finite  by definition,
$\gr_{\mathcal F} M_D$ is generic free by Lemma \ref{xxlem4.4}(2).
By replacing $D$ by $D[s^{-1}]$ for some $s$, we assume that
$\gr_{\mathcal F} M_D$ is free over $D$. As a consequence, each $F_i M_D$
is free of finite rank over $D$. In this case,
the $\GKdim$ of $M_D$ is only dependent on the
Hilbert series of $\gr_{\mathcal F} M_D$. Similar comment holds true
for $M$. One can check that $\gr_{\mathcal F} M_D$ is an order of $\gr_{\mathcal F} M$,
whence, the Hilbert series of $\gr_{\mathcal F} M_D$ over $D$ is
equal to the Hilbert series of $\gr_{\mathcal F} M$ over $\Bbbk$.
Therefore
$$\GKdim (M_D)_{A_D}=\GKdim (\gr_{\mathcal F} M_D)_{\gr_{\mathcal F} A_D}
=\GKdim (\gr_{\mathcal F} M)_{\gr_{\mathcal F} A}=\GKdim M_A.$$
Now consider an algebra map $D\to F$ where $F$ is a noetherian
commutative domain. Since $\gr_{\mathcal F} M_D$ is free over $D$,
$(\gr_{\mathcal F} M_D)\otimes_D F$ is free over $F$. Consequently,
$M_F:=M_D\otimes_D F$ is free over $F$. The $\GKdim$ of $M_F$
can be computed by the Hilbert series of $\gr_{\mathcal F} M_F$ since
$\gr_{\mathcal F} M_F$ is a f.g. right module over the noetherian graded ring
$\gr_{\mathcal F} A_F$. Therefore
$$\GKdim (M_{F})_{A_{F}}=\GKdim (M_{D})_{A_D}=\GKdim M_A,$$
and thus $M$ is $\GKdim$-stable over $D$.
\end{proof}

\begin{proposition}
\label{xxpro4.8}
Let $R$ be an congenial algebra with a filtration ${\mathcal F}$.
Let $H$ be a semisimple Hopf algebra acting on $R$
with $H$-action preserving the filtration.
Let $B:=R\# H$ and $A:=eBe$ as in Hypothesis {\rm{\ref{xxhyp2.4}}}.
Then the following hold.
\begin{enumerate}
\item[(1)]
$B$ is congenial.
\item[(2)]
$A$ is congenial.
\item[(3)]
Let $M$ be $R(=eB)$.
There is a $\Bbbk_0$-affine subalgebra $D\subseteq \Bbbk$
and an order $M_D$ of $M$ such that $M_D$ is an
$(A_D, B_D)$-bimodule that is $D$-central and f.g. on both sides.
\end{enumerate}
\end{proposition}

\begin{proof}
(1) Since $R$ is congenial, there is an order $R_D$ satisfying
the conditions listed in Definition \ref{xxdef4.3}. Remember
that we can freely replace $D$ by a $D$-affine subalgebra
$D'\subseteq \Bbbk$.

Since $H$ is finite dimensional with normalized integral $\inth$,
there is an order $H_D$ of $H$ such that $H_D$ is a Hopf $D$-algebra
and $\inth\in H_D$. Then $R_D\# H_D$ is an order of $R\# H$.
The filtration on $R$ extends to a filtration on $R\# H$ by setting
elements of $H$ in $F_0$. Then Definition \ref{xxdef4.3}(1,2) holds
for the algebra $R\# H$.

Since $\gr_{\mathcal F} R$ is congenial, we can assume that $F_i (R_D)$ and
$F_i (R_D)/ F_{i-1} (R_D)$ are free over $D$ for all $i$ because
$\gr_{\mathcal F} R_D$ is free over $D$. Thus every $F_i (R_D\# H_D)$ and
$F_i (R_D\# H_D)/ F_{i-1} (R_D\# H_D)$ are free over $D$ for all $i$.
Therefore we have $\gr_{\mathcal F} (R_D\# H_D)=(\gr_{\mathcal F} R_D)\# H_D$. Since
$\gr_{\mathcal F} R_D$ satisfies Definition \ref{xxdef4.3}(3,4), so does
$(\gr_{\mathcal F} R_D)\# H_D$. Therefore $\gr_{\mathcal F} (R_D\# H_D)$
satisfies Definition \ref{xxdef4.3}(3,4). Note that Definition \ref{xxdef4.3}(5)
is clear for $B$. The assertion follows
by combining the last two paragraphs.

(2) We identify $A$ with $R^H$, together with the filtration induced by
${\mathcal F}$. Then
\begin{equation}
\label{E4.8.1}\tag{E4.8.1}
\gr_{\mathcal F} A= \gr_{\mathcal F} (R^H)
=(\gr_{\mathcal F} R)^H=\inth\cdot (\gr_{\mathcal F} R)
\end{equation}
by \cite[Lemma 3.1]{CWWZ}.
Define $A_D:=(R_D)^{H_D}=\inth \cdot R_D$ with
the filtration induced by the filtration on $R$. Recall that $R_D$ is
strongly noetherian, so is $A_D$ by \cite[Lemma 3.1]{BHZ}.
Since $R$ is congenial, the hypothesis of \cite[Theorem 0.2]{ArSZ}
hold, and whence, we can assume that $A_D$ is free over $D$ after
changing $D$ if necessary. The equation $A_D=\inth \cdot R_D$ also shows that
$A_D$ is an order of $A$. Then
Definition \ref{xxdef4.3}(1,2) hold for $A$.

Similarly, using \eqref{E4.8.1}, one
can check that Definition \ref{xxdef4.3}(3,4) hold for $A$.
Definition \ref{xxdef4.3}(5)
is clear for $A=\inth \cdot R$. Therefore
the assertion follows.

(3) Let $D$ be such that $B_D$ (respectively, $H_D$ and $R_D$)
is an order of $B$ (respectively, $H$ and $R$ sitting inside $B$).
Identify $eB$ (respectively, $eBe$) with $R$ (respectively,
$A$) by \cite[Lemma 3.1(4,2)]{BHZ}. Since $A$ is congenial,
by Lemma \ref{xxlem4.4}(2), one can assume that $eB_D =R_D$
(respectively, $eB_De$) is an order of $eB=R$ (respectively,
$eBe$). By \cite[Lemma 3.1(2)]{BHZ}, $eB_D$ is f.g. in both sides.
Finally it is clear that $eB_D $ is $D$-central. The assertion
holds.
\end{proof}

\begin{theorem}
\label{xxthm4.9}
Let $R$ be a congenial algebra with a filtration $\mathcal{F}$.
Let $H$ be a semisimple Hopf algebra acting on $R$
with $H$-action preserving the filtration.
Let $B=R\# H$ and $A=eBe$ as in Hypothesis {\rm{\ref{xxhyp2.4}}}.
Then $\gamma_{n,j}(eB)$ holds for all $n,j$.
\end{theorem}

\begin{proof} It suffices to show
that, for every right $A$-module $N$,
$$\GKdim (\Tor^A_{j}(N, eB))_B \leq \GKdim N_A.$$
By Lemma \ref{xxlem4.4}(2) and Proposition
\ref{xxpro4.7}, we can pick a $D$ such that all of $N_A$, $_A(eB)$,
$(eB)_B$ and $(\Tor^A_{i}(N, eB))_B$, for $i\leq j$, are free and
$\GKdim$-stable over $D$. Then, for every algebra homomorphism
$D\longrightarrow F$ of commutative rings,
$$\begin{aligned}
\GKdim (\Tor^A_{j}(N, eB))_B&= \GKdim (\Tor^A_{j}(N, eB)\otimes_D F)
\qquad {\text{ Definition \ref{xxdef4.5}}}\\
&= \GKdim (\Tor^{A_F}_{j}(N_F, (eB)_F))
\qquad \quad{\text{Lemma  \ref{xxlem4.2}(2).}}
\end{aligned}
$$
Note that Lemma \ref{xxlem4.2}(2) can be applied in the last equation 
since all modules involved are free over $D$.
By Definition \ref{xxdef4.3}(5), for every finite field quotient
$F$ of $D$, $R_F$ is noetherian affine PI. In this case, $\Kdim
M=\GKdim M$ for all f.g modules $M$ over $R_F$, (or over $A_F$,
$B_F$). Then
$$\begin{aligned}
\GKdim (\Tor^{A_F}_{j}(N_F, (eB)_F)) &=
\Kdim (\Tor^{A_F}_{j}(N_F, (eB)_F)) \\
&\leq \Kdim N_F \qquad\qquad\qquad\qquad \quad {\text{Lemma \ref{xxlem3.1}(2)}}\\
&=\GKdim N_F= \GKdim N_B \qquad{\text{Definition \ref{xxdef4.5}.}}
\end{aligned}
$$
Therefore the assertion follows.
\end{proof}

Now we have an Auslander theorem for congenial algebras as a consequence
of Theorems \ref{xxthm2.3} and \ref{xxthm4.9}.

\begin{theorem}
\label{xxthm4.10}
Let $R$ be a congenial algebra with a filtration ${\mathcal F}$.
Let $H$ be a semisimple Hopf algebra acting on $R$
with $H$-action preserving the filtration.
Let $B=R\# H$ and $A=eBe$ as in Hypothesis {\rm{\ref{xxhyp2.4}}}
and suppose that $R$ is CM. Then the following are equivalent:
\begin{enumerate}
\item[(1)]
There is a natural isomorphism of algebras $R\# H\cong \End_{R^H}(R)$.
\item[(2)]
$\Pty(R,H)\geq 2$.
\item[(3)]
The $H$-action on $A$ is  h.small.
\end{enumerate}
\end{theorem}

There are ample examples of congenial algebras, one of which is
the universal enveloping algebra of any finite dimensional
Lie algebra.

\begin{lemma}
\label{xxlem4.11}
The universal enveloping algebra
$R:=U(\mathfrak g)$ of a finite dimensional
Lie algebra ${\mathfrak g}$ is congenial.
\end{lemma}

\begin{proof} We need to show (1-5) in Definition \ref{xxdef4.3}.

(1) $R$ has a standard filtration ${\mathcal F}$ defined by
$F_i R=(\Bbbk+{\mathfrak g})^i$ for all $i\geq 0$. It is well-known
that $R$ is a noetherian locally finite filtered algebra.

(2) Pick any $\Bbbk$-linear basis $\{x_1,\cdots,x_n\}$ of the Lie algebra
${\mathfrak g}$ and consider the Lie product
$[x_i, x_j]=\sum_{s=1}^n c_{i,j}^s x_s$ with coefficients
$c_{i,j}^s\in \Bbbk$. Let $D$ be the $\Bbbk_0$-subalgebra of
$\Bbbk$ generated by these $c_{i,j}^s$.

Let ${\mathfrak g}_D=\bigoplus_{s=1}^n D x_i$. Then ${\mathfrak g}_D$
is a $D$-Lie algebra with ${\mathfrak g}_D\otimes_D \Bbbk={\mathfrak g}$.
One can define the universal enveloping algebra $R_D:=U({\mathfrak g}_D)$
to be
$$D\langle x_1,\cdots,x_n\rangle/([x_i, x_j]=\sum_{s=1}^n c_{i,j}^s x_s,\forall
i,j).$$
Then $U({\mathfrak g}_D)$ has a filtration ${\mathcal F}'$ which is
compatible with the filtration ${\mathcal F}$ such that $\gr_{{\mathcal F}'}
R_D$ is isomorphic to the commutative polynomial ring
$D[x_1,\cdots,x_n]$. Hence, $R_D$ is free over $D$ and $R_D$ is an order
of $R$. Therefore (2) holds.

The proof of (2) also shows that $\gr_{{\mathcal F}'} R_D$ is an order
of $\gr_{{\mathcal F}} R$, which is (3).

(4) This follows from the fact that $\gr_{{\mathcal F}'} R_D$ is the
commutative polynomial ring over $D$.

(5) One can check that $R_D\otimes_D F$ is the universal enveloping algebra
$U({\mathfrak g}_D\otimes_D F)$. Since $F$ is a finite field, 
${\rm{char}}\; F>0$. By \cite{Ja}, $U({\mathfrak g}_D\otimes_D F)$
is a finite module over its affine center. Therefore it is an affine 
noetherian PI domain over $F$.
\end{proof}

In order to prove Theorem \ref{xxthm0.4}, we need a further lemma.

\begin{lemma}
\label{xxlem4.12}
Let $V$ be a finite dimensional $\Bbbk$-vector space 
of dimension at least $2$,  and
$G$ be a finite subgroup of ${\text{GL}}(V)$
acting on the polynomial ring $R:=\Bbbk[V]$
naturally. Then the following are equivalent.
\begin{enumerate}
\item[(1)]
$G\subseteq {\rm{GL}}(V)$ is small.
\item[(2)]
The natural algebra map $R\ast G\to \End_{R^G}(R)$ is an
isomorphism of algebras.
\item[(3)]
$\Pty(R,G)\geq 2$.
\item[(4)]
The $G$-action on $R$ is h.small.
\end{enumerate}
\end{lemma}

\begin{proof} (1) $\Longrightarrow$ (2) This is the commutative
Auslander theorem \cite{Au1, Au2}. See \cite[Theorem 4.2]{IT}
for a more recent proof.

(2) $\Longrightarrow$ (1) Suppose $G$ is not small. Let $W$ be the subgroup
of $G$ generated by a pseudo-reflection in $G$. Then $W$ is nontrivial and
$R^W$ is a commutative polynomial ring by Shephard-Todd-Chevalley theorem.
In this case $R$ is a free module over $R^W$ of rank at least 2 and write
$R\cong A^{\oplus d}$ where $A=R^W$ and $d\geq 2$. Since
$W$ is a subgroup of $G$, $R^G\subseteq A$. Thus $\End_{R^G}(R)\cong \End_{R^G}(A^{\oplus d})
=M_d(\End_{R^G}(A))$. This implies that every factor ring of
$\End_{R^G}(R)$ has dimension at least $d^2>1$. However, $R\ast G$ has a factor
ring $\Bbbk$ by sending $g\mapsto 1$ and $R_{\geq 1}\mapsto 0$. Therefore
$R\ast G\not\cong \End_{R^G}(R)$. The assertion follows.

(2) $\Leftrightarrow$ (3) This follows from Theorem \ref{xxthm2.7}.

(3) $\Leftrightarrow$ (4) This follows from  the fact that $\Bbbk[V]$
is CM.
\end{proof}

\begin{proof}[Proof of Theorem \ref{xxthm0.4}]
As in the proof of Lemma \ref{xxlem4.11}, $R$ is congenial and
has a standard filtration ${\mathcal F}$. The $G$-action on $R$
induces naturally a $G$-action on $\gr_{\mathcal F} R\cong
\Bbbk[x_1,\cdots,x_n]$. Then, by Proposition \ref{xxpro3.6} and Lemma \ref{xxlem4.12},
$$\Pty(R,G)\geq \Pty(\gr_{\mathcal F} R, G)\geq 2.$$
 It is well known
that $U({\mathfrak g})$ is CM. The assertion follows from
Theorem \ref{xxthm4.10}.
\end{proof}

To prove Corollary \ref{xxcor0.5} we need the following easy
observation.

\begin{lemma}
\label{xxlem4.13}
Let ${\mathfrak g}$ be a finite dimensional Lie
algebra such that it is not a semidirect product
${\mathfrak g}'\ltimes \Bbbk x$ for a $1$-dimensional
ideal $\Bbbk x$. Then every finite subgroup $G\subseteq
\Aut_{Lie}({\mathfrak g})$ is small.
\end{lemma}

\begin{proof} If $G$ is not small, there is a $1\neq g\in G$
which is a pseudo-reflection. Then there is a  decomposition
${\mathfrak g}=V\oplus \Bbbk x$ such that $g\mid_{V}$ is the
identity and $g(x)=a x$ for some $a\neq 1$. Then it is easy
to show that $V$ is  a Lie subalgebra of
${\mathfrak g}$ and $\Bbbk x$ is a 1-dimensional
Lie ideal of ${\mathfrak g}$.
Therefore ${\mathfrak g}=V\ltimes \Bbbk x$, a contradiction.
\end{proof}

\begin{proof}[Proof of Corollary 0.5] Under the hypothesis,
every finite subgroup $G\subseteq \Aut_{Lie}({\mathfrak g})$
is small by Lemma \ref{xxlem4.13}. The assertion follows
from Theorem \ref{xxthm0.4}.
\end{proof}

\section{Twisting and skew polynomial rings}
\label{xxsec5}

In the rest of the paper we assume that $\Bbbk$ is a field.

Let $\Gamma$ be a group and let $A$ be an ${\mathbb N}\times
\Gamma$-graded algebra. Assume that $A$ is locally finite when considered
as an ${\mathbb N}$-graded algebra. We first recall
$\Gamma$-twisting systems and twisted algebras from
\cite{Zh2}.

\begin{definition}
\label{xxdef5.1}
Let $A$ be an ${\mathbb N}\times \Gamma$-graded algebra.
A set of ${\mathbb N}\times \Gamma$-graded algebra automorphisms of $A$, denoted by
$\tau:=\{\tau_{\gamma}\mid \gamma\in \Gamma\}$, is called
a {\it twisting system} of $A$ if
$$\tau_{\gamma_1}\tau_{\gamma_2}=\tau_{\gamma_1\gamma_2}$$
for all $\gamma_1,\gamma_2\in \Gamma$.
\end{definition}

The original definition of a twisting system  \cite[Definition 2.1]{Zh2}
is slightly more general.
Given a twisting system, we define a new multiplication of
$A$ by
$$ x \ast y=x\tau_{\gamma} (y)$$
if $x$ is homogeneous of degree $(n,\gamma)$. By
\cite[Proposition and Definition 2.3]{Zh2}, the {\it twisted
algebra} $A^{\tau}:= (A,\ast)$ is another ${\mathbb N}\times
\Gamma$-graded algebra and $A^{\tau}=A$ as a graded
$\Bbbk$-space.

\begin{example}
\label{xxex5.2}
Let $\{p_{ij}\mid 1\leq i<j \leq n\}$ be a subset of
$\Bbbk^{\times}$ and let $A$ be the skew polynomial ring
$\Bbbk_{p_{ij}}[x_1,\cdots,x_n]$ generated by
$x_1,\cdots,x_n$ and subject to the relations
\begin{equation}
\label{E5.2.1}\tag{E5.2.1}
x_j x_i=p_{ij}x_i x_j
\end{equation}
for all $i<j$.  Define $\deg x_i=(1, e_i)$ where
$e_i$ is the $i$th unit vector in ${\mathbb Z}^n$.
Then $A$ is ${\mathbb N}\times \Gamma$-graded
where $\Gamma={\mathbb Z}^n$. By \cite[p. 310]{Zh2},
$A$ is a twisted algebra of the commutative polynomial
ring $\Bbbk[x_1,\cdots,x_n]$.
\end{example}

Now we put a Hopf action into the picture.

\begin{hypothesis}
\label{xxhyp5.3}
\begin{enumerate}
\item[(1)]
Let $A$ be an ${\mathbb N}\times \Gamma$-graded algebra
with a twisting system $\tau=\{\tau_{\gamma}\mid \gamma\in \Gamma\}$.
\item[(2)]
Let $H$ be a semisimple Hopf algebra acting on $A$ and
preserving the ${\mathbb N}\times \Gamma$-grading.
\item[(3)]
The $H$-action commutes with $\tau$ in the sense that
$$h \cdot (\tau_{\gamma}(x))=\tau_{\gamma} (h\cdot x)$$
for all $h\in H$, $\gamma\in \Gamma$ and $x\in A$.
\end{enumerate}
\end{hypothesis}

The following lemma is easy and we skip most of the proof.

\begin{lemma}
\label{xxlem5.4}
Retain Hypothesis {\rm{\ref{xxhyp5.3}}}.
\begin{enumerate}
\item[(1)]
The $H$-action on $A$ induces naturally an $H$-action on $A^{\tau}$
which preserves the grading.
\item[(2)]
The smash product $A\# H$ is an ${\mathbb N}\times \Gamma$-graded
algebra where $\deg (1\#h)=(0,1_{\Gamma})$ for all $h\in H$.
\item[(3)]
The $\tau$ extends to a twisting system $\tau'$ of
$A\# H$ by $\tau'_{\gamma}(x\# h)=\tau_{\gamma}(x)\# h$
for all $x\in A$ and $h\in H$, and $(A\# H)^{\tau'}=(A^{\tau})\# H$.
\item[(4)]
Let $e=1\# \inth$. Then $\tau'_{\gamma}(e)=e$ for all
$\gamma\in \Gamma$. As a consequence, $\tau'$ induces
a twisting system $\tau''$ of $(A\# H)/(e)$ such that
$$((A\#H)/(e))^{\tau''}\cong (A\#H)^{\tau'}/(e)\cong (A^{\tau}\# H)/(e)$$
as graded algebras.
\item[(5)]
Assume that both $A$ and $A^{\tau}$ are noetherian and locally finite
as ${\mathbb N}$-graded algebras.
Then $\Pty(A,H)=\Pty(A^{\tau},H)$.
\end{enumerate}
\end{lemma}

\begin{proof} (5) Since $A$ and $A^{\tau}$ are noetherian and locally finite
as ${\mathbb N}$-graded algebras, we can use \eqref{E1.1.6} to compute
the GK-dimension, this means that $\GKdim$ is independent of their algebra
structure. In particular, $\GKdim A=\GKdim A^{\tau}$ and
$\GKdim (A\#H)/(e)=\GKdim (A^{\tau}\# H)/(e)$ by part (4). The assertion
follows.
\end{proof}

Now we are ready to prove Theorem \ref{xxthm0.8}. An automorphism
$\phi$ of $\Bbbk_{p_{ij}}[x_1,\cdots,x_n]$ is called {\it diagonal}
if $\phi(x_i)=a_i x_i$ for some $a_i\in \Bbbk^{\times}$.

\begin{theorem}
\label{xxthm5.5}
Suppose ${\rm{char}}\; \Bbbk=0$.
Let $R$ be the ring $\Bbbk_{p_{ij}}[x_1,\cdots,x_n]$
for $n\geq 2$ and $\{p_{ij}\}\subseteq \Bbbk^{\times}$.
Let $G$ be a finite group of diagonal algebra automorphisms of $R$.
Then the following are equivalent.
\begin{enumerate}
\item[(1)]
$G\subseteq {\rm{GL}}(V)$ is small where $V=\bigoplus_{s=1}^n \Bbbk x_s$.
\item[(2)]
The $G$-action on $R$ is  h.small.
\item[(3)]
$\Pty(R,G)\geq 2$.
\item[(4)]
There is a natural isomorphism of algebras
$R\ast G\cong\End_{R^G}(R).$
\end{enumerate}
\end{theorem}

\begin{proof} (2) $\Leftrightarrow$ (3) $\Leftrightarrow$ (4) Follow
from \cite[Theorem 3.5]{BHZ}.

(1) $\Leftrightarrow$ (2)
By \cite[p.310]{Zh2}, $\Bbbk_{p_{ij}}[x_1,\cdots,x_n]$
is a twisted algebra of the commutative polynomial ring
by $\tau$ where $\tau$ consisting of diagonal algebra
automorphisms.

Let $H=\Bbbk G$. Then one can easily check Hypothesis \ref{xxhyp5.3}.
By Lemma \ref{xxlem5.4}(5), $\Pty(R,H)=\Pty(R^{\tau^{-1}},H)$
where $R^{\tau^{-1}}$ is the commutative polynomial ring.
The smallness of $G$ is only dependent on the $G$-action
on the vector space $V=\bigoplus_{s=1}^n \Bbbk x_s$. So we
can assume that $R$ is the commutative polynomial ring
$R^{\tau^{-1}}$. Therefore the assertion follows by Lemma
\ref{xxlem4.12}.
\end{proof}

\begin{proof}[Proof of Theorem \ref{xxthm0.8}]
When $p_{ij}$ are generic, every algebra automorphism
is diagonal by \cite{AlC}. The assertion follows by Theorem
\ref{xxthm5.5}.
\end{proof}

\section{Down-up algebras}
\label{xxsec6}

First we recall a definition, and refer to \cite{BR, KK, KMP}
for basic properties of down-up algebras.

\begin{definition} \cite{BR}
\label{xxdef6.1}
A {\it graded down-up algebra}, denoted by $A(\alpha, \beta)$,
with parameters $\alpha, \beta\in \Bbbk$, is generated by
$x$ and $y$, and subject to the relations
$$\begin{aligned}
x^2y&=\alpha xyx +\beta y x^2,\\
xy^2&=\alpha yxy +\beta y^2 x.
\end{aligned}
$$
\end{definition}

The universal enveloping algebra of the $3$-dimensional
Heisenberg Lie algebra is $A(2,-1)$, and there are other
interesting special cases, see \cite{BR, KK}. It is
well-known that $A(\alpha, \beta)$ is noetherian if
and only if it is AS regular if and only if
$\beta\neq 0$ \cite{KMP}.

Let $R$ be a noetherian graded
down-up algebra $A(\alpha, \beta)$ where $\beta\neq 0$.
We let $a$ and $b$ be the roots of
the ``character polynomial''
$$t^2 -\alpha t - \beta=0.$$
Then
$$ \Omega_1:= xy-a yx, \quad {\text{and}}\quad
\Omega_2:= xy-b yx$$
are normal regular elements in $R$. We recall
a result of Kirkman-Kuzmanovich.

\begin{lemma}\cite[Proposition 1.1]{KK}
\label{xxlem6.2}
Let $R$ be $A(\alpha, \beta)$. Then the group
of graded algebra automorphisms of $R$ is given by
$$
\Aut_{gr}(R)=\begin{cases}
{\rm{GL}}(\Bbbk^{\oplus 2}) & {\text{if }} (\alpha,\beta)=(0,1),\\
{\rm{GL}}(\Bbbk^{\oplus 2}) & {\text{if }} (\alpha,\beta)=(2,-1),\\
U:=\left\{ \begin{pmatrix} a_{11} &0\\0&a_{22}\end{pmatrix},
\begin{pmatrix} 0 &a_{12}\\a_{21}&0 \end{pmatrix}: a_{ij}\in \Bbbk^{\times}
\right\} & {\text{if }} \beta=-1,\alpha\neq 2,\\
O:=\left\{ \begin{pmatrix} a_{11} &0\\0&a_{22}\end{pmatrix}: a_{ij}\in \Bbbk^{\times}
\right\} & {\text{otherwise.}}
\end{cases}
$$
As a consequence, $\Aut_{gr}(R)$ can be realized as a subgroup of
${\rm{GL}}(\Bbbk^{\oplus 2})$.
\end{lemma}

\begin{lemma}
\label{xxlem6.3}
Let $R=A(\alpha,\beta)$ and suppose that $\beta\neq -1$ or
$(\alpha,\beta)=(2,-1)$. Let $G$ be a subgroup $\Aut_{gr}(R)$
sitting inside ${\rm{GL}}(\Bbbk^{\oplus 2})$ -- see Lemma {\rm{\ref{xxlem6.2}}}.
Then there is a normal regular element $\Omega=xy-a yx$ such that
$g(\Omega)=\det(g) \Omega$ for all $g\in G$.
\end{lemma}

\begin{proof} There are four cases according to
Lemma \ref{xxlem6.2}. The hypotheses rule out the
third case. In the first two cases
($(\alpha,\beta)=(0,1)$ or $(2,-1)$), we take
$a=1$. It is easy to check that $g(\Omega)=
\det(g)\Omega$
for all $g\in \Aut_{gr}(R)$. In the fourth
case, we take $a$ to be any of the roots of the
character polynomial. Since $g$ is represented
by a diagonal matrix, $g(\Omega)=\det(g)\Omega$ when
$g$ is in $\Aut_{gr}(R)$.
\end{proof}

When $\beta=-1$ and $\alpha\neq 2$, and if $g$ is
represented by a matrix of the form
$\begin{pmatrix} 0 &a_{12}\\a_{21}&0 \end{pmatrix}$,
then $g$ switches $xy-ayx$ and $xy-byx$ where $a$
and $b$ are two different roots of the character polynomial.
This is the case that we can not handle.

We are now ready to show Theorem \ref{xxthm0.6}.

\begin{proof}[Proof of Theorem \ref{xxthm0.6}]
It is well-known that $R$ is CM. By Theorem
\ref{xxthm2.7}, it suffices to show that $\Pty(R,G)\geq 2$.
The idea of the proof is that we construct some filtration of the
down-up algebra $R$ and apply Proposition \ref{xxpro3.6}.

Let $R$ be a down-up algebra $A(\alpha,\beta)$ in the situation
of Lemma \ref{xxlem6.3}, where
$\beta\neq -1$ if $\alpha\neq 2$. By Lemma \ref{xxlem6.3},
this algebra is generated
by $x$ and $y$ and having a normal element $\Omega:=xy-ayx$
such that $g(\Omega)=\det(g) \Omega$ for all $g\in G$.
Consider $R$ as an ungraded algebra and define
a filtration ${\mathcal F}$ by setting
$$F_n R=(\Bbbk\oplus \Bbbk x\oplus \Bbbk y\oplus \Bbbk \Omega)^n\subseteq R$$
for all $n\ge0$ \cite[Lemma 7.2(2)]{KKZ1}.
Then $F_n R$ is $G$-stable, consequently,
$A:=\gr_{\mathcal F} R$ is a connected graded algebra with $G$-action.
As a $G$-module, $A$ is isomorphic to $R$, then the $G$-action
on $A$ is inner faithful and homogeneous. We will see soon that
$A$ is noetherian. By Proposition \ref{xxpro3.6}, it suffices to show
that $\Pty(A,G)\geq 2$.

Let $x$ and $y$ denote the corresponding elements of $x$ and $y$
in $A$, and let $z$ be the image of $\Omega$ in $A$. Since
$\Omega$ is normal in $R$, we obtain an isomorphism of graded
algebras:
\begin{equation}\label{E6.3.1}\tag{E6.3.1}
A:=\gr_{\mathcal F}(R)\cong (\Bbbk_{a^{-1}}[x,y])[z;\sigma],
\end{equation}
where $\sigma$ is a graded algebra isomorphism of $\Bbbk_{a^{-1}}[x,y]
=R/(\Omega)$ determined by the equations $\Omega x=\sigma(x)\Omega$
and $\Omega y=\sigma(y)\Omega$ in $R$. Let $V'$ be the $\Bbbk$-subspace
$\Bbbk x+\Bbbk y+\Bbbk z$ of $A$. For $g\in G$, we have seen $g(z)=\det (g) z$. 
Since $G$ is a finite group, there is a basis $\{x',y'\}$ of 
$\Bbbk x+\Bbbk y$ such that $g(x')=\xi_1 x'$ and $g(y')=\xi_2 y'$, 
where $\xi_1$ and $\xi_2$ are roots of unity. Then $g(z)=\xi_1\xi_2 z$. 
Hence, for any $1\neq g\in G$, at least two of the three numbers
$\{\xi_1, \xi_2, \xi_1\xi_2\}$ 
are not 1, that is, $g$ is not a pseudo-refection. Therefore $G$ is small.
Next we prove that $\Pty(A, G)\geq 2$ in the following three cases.

Case 1: $(\alpha, \beta)=(0,1)$, which is the first case in Lemma
\ref{xxlem6.2}. Then $a=1$ and $\sigma: x\mapsto -x, y\mapsto -y$.
In this case, $A$ is a graded twist of $\Bbbk[x,y][z]$ with a
twisting system commuting with the $G$-action. By Lemma \ref{xxlem5.4}(5),
it suffices to show that $\Pty(A^{\tau}, G)\geq 2$ where
$A^{\tau}$ is the commutative polynomial ring. When $A^\tau$
is the commutative polynomial ring, $\Pty(A^{\tau}, G)\geq 2$
is equivalent to the fact that $G$ is small by Lemma \ref{xxlem4.12}.

Case 2: $(\alpha,\beta)=(2,-1)$, which is the second case in Lemma
\ref{xxlem6.2}. Then $A$ is the commutative polynomial ring
$\Bbbk[x,y,z]$. The assertion follows by Lemma \ref{xxlem4.12}.

Case 3: The ``otherwise'' case in Lemma \ref{xxlem6.2}.
In this case, every $g\in G$ is diagonal and $A$ is a skew polynomial
ring $\Bbbk_{p_{ij}}[x,y,z]$. As a consequence, $g$ acts
on $A$ diagonally. Then the assertion follows from
Theorem \ref{xxthm5.5}.
\end{proof}

\section{Some comments on smallness}
\label{xxsec7}
In this section we provide some easy examples and
comments on different definitions of smallness.
For simplicity, assume that ${\rm{char}}\; \Bbbk=0$.

Recall that a finite subgroup $G$ of the general
linear group ${\text{GL}}(V)$ is called {\it small}
if $G$ does not contain a pseudo-reflection of $V$
(except for the identity).
We now recall the definition of conventionally small.

Let $g$ be a graded algebra automorphism of a connected graded
algebra $R$. Recall from \cite{JZ} that the {\it trace function}
of $g$ is defined to be
$$Tr_R(g, t):=\sum_{i=0}^{\infty} \tr(g\mid_{R_i}) t^i\in
\Bbbk[[t]].$$

\begin{definition}\cite[Definition 0.8]{BHZ}
\label{xxdef7.1}
Let $g$ be a graded algebra automorphism of a noetherian Koszul
AS regular algebra $R$ of finite order
and let $G$ be a finite subgroup of $\Aut_{gr}(R)$.
\begin{enumerate}
\item[(1)]
The {\it reflection number} of $g$ is defined to be
$$\Rpf(g):=\GKdim R-{\text{the order of the pole of $Tr_R(g,t)$ at $t=1$.}}$$
\item[(2)]
\cite[Definition 2.2]{KKZ3}
$g$ is called a {\it quasi-reflection} if $\Rpf(g)=1$.
\item[(3)]
\cite[Definition 3.6(b)]{KKZ4}
$g$ is called a {\it quasi-bireflection} if $\Rpf(g)=2$.
\item[(4)]
The {\it reflection number} of $G$-action on $R$ is defined to be
$$\Rpf(R,G):=\min \{ \Rpf(g)\mid 1\neq g\in G\}.$$
\item[(5)]
The $G$-action on $R$ is called {\it conventionally small} or {\it c.small}
if $\Rpf(R,G)\geq 2$, or equivalently, $G$ does not contain any quasi-reflection.
\end{enumerate}
\end{definition}

In view of Lemma \ref{xxlem4.12}, we have the following equivalences of different smallness.

\begin{lemma}
\label{xxlem7.2}
Let $V$ be a finite dimensional $\Bbbk$-vector space of dimension at least $2$,  and
$G$ be a finite subgroup of ${\text{GL}}(V)$
acting on the polynomial ring $R:=\Bbbk[V]$
naturally. Then the following are equivalent.
\begin{enumerate}
\item[(1)]
$G\subseteq {\rm{GL}}(V)$ is small.
\item[(2)]
$\Pty(R,G)\geq 2$.
\item[(3)]
The $G$-action on $R$ is h.small.
\item[(4)]
The $G$-action on $R$ is c.small.
\end{enumerate}
\end{lemma}

\begin{proof}
(1) $\Leftrightarrow$ (4) For polynomial algebra $\Bbbk[V]$,
a quasi-reflection in Definition \ref{xxdef7.1}(2)
agrees with the classical definition of a pseudo-reflection.
So the assertion follows. See Lemma \ref{xxlem4.12} for other parts.
\end{proof}

\begin{example}
\label{xxex7.3}
Let $R$ be the skew polynomial ring
$\Bbbk_{-1}[x,y]$ generated by $x$ and $y$ subject to
the relations $xy=-yx$. Let $G$ be the group of algebra automorphisms
of $R$ generated by $\sigma: x\mapsto y, y\mapsto x$. We claim the
following:
\begin{enumerate}
\item[(1)]
$G\subseteq {\rm{GL}}(V)$ is NOT small, where $V=\Bbbk x+\Bbbk y$.
\item[(2)]
The $G$-action on $R$ is h.small.
\item[(3)]
The $G$-action on $R$ is c.small.
\end{enumerate}

(1) is obvious. (2) is a special case of \cite[Theorem 0.5]{BHZ}.
(3) Follows from the formula in \cite[Lemma 1.7(1)]{KKZ2}.

If $B$ is the commutative polynomial ring
$\Bbbk[x,y]$ and $G$ is the group of algebra automorphisms of
$B$ generated by $\sigma: x\mapsto y, y\mapsto x$. Then we have the following
by Lemma \ref{xxlem7.2}:
\begin{enumerate}
\item[(4)]
$G\subseteq {\rm{GL}}(V)$ is not small, where $V=\Bbbk x+\Bbbk y$.
\item[(5)]
The $G$-action on $B$ is not h.small.
\item[(6)]
The $G$-action on $B$ is not c.small.
\end{enumerate}
\end{example}

\begin{example}
\label{xxex7.4}
Let $R:=\Bbbk_{-1}[x,y]$ be as in Example \ref{xxex7.3} and $G'$
be the group of algebra automorphisms of $R$
generated by $\sigma: x\mapsto i y, y\mapsto i x$ where $i^2=-1$. By a $\Bbbk$-linear
base change, we are in the situation of \cite[Example 2.3(b)]{KKZ3}. We claim
the following:
\begin{enumerate}
\item[(1)]
$G'\subseteq {\rm{GL}}(V)$ is small, where $V=\Bbbk x+\Bbbk y$.
\item[(2)]
The $G'$-action on $R$ is NOT h.small.
\item[(3)]
The $G'$-action on $R$ is NOT c.small.
\end{enumerate}

Indeed, the statement (1) is Obvious.

(2) By \cite[Example 2.3(b)]{KKZ3}, $A:=R^{G'}$ is AS regular.
Then $R$ is a free module over $A$ of rank $d\geq 2$.
So $\End_A(R)$ is the $d\times d$-matrix algebra over $A$.
Hence, $\End_A(R)$ can not be isomorphic to $R\ast G'$.
By Theorem \ref{xxthm2.7}, the $G'$-action on $R$ is NOT h.small.

(3) It follows from the formula in \cite[Lemma 1.7(1)]{KKZ2} that
$Tr_R(\sigma,t)=\frac{1}{(1-t^2)}$. Or one can use the
trace formula given in \cite[Example 2.3(b)]{KKZ3}. Then the
$G'$-action on $R$ is NOT c.small.

If $B$ is the commutative polynomial ring
$\Bbbk[x,y]$ and $G'$ is the group of algebra automorphisms of
$A$ generated by $\sigma: x\mapsto iy, y\mapsto ix$. Then we have the following
by Lemma \ref{xxlem7.2}
\begin{enumerate}
\item[(4)]
$G'\subseteq {\rm{GL}}(V)$ is small, where $V=\Bbbk x+\Bbbk y$.
\item[(5)]
The $G'$-action on $B$ is h.small.
\item[(6)]
The $G'$-action on $B$ is c.small.
\end{enumerate}
\end{example}

\begin{remark}
\label{xxrem7.5} Let $R$ be a noetherian AS regular, CM  and
Koszul algebra.
\begin{enumerate}
\item[(1)]
The above two examples show that the smallness is different from
the h.smallness in the general noncommutative setting.
\item[(2)]
We conjecture that the c.smallness is equivalent to the h.smallness.
\item[(3)]
By ``(2) $\Longrightarrow$ (1)'' in the proof of Lemma \ref{xxlem4.12},
h.smallness is stronger than c.smallness.
Therefore the conjecture in part (2) follows from
\cite[Conjecture 0.9]{BHZ}.
\item[(4)]
Let $R$ be the commutative polynomial ring as in Lemma \ref{xxlem7.2}.
By using the definition of smallness, one sees that the Auslander Theorem
holds if and only if
\begin{equation}
\label{E7.5.1}\tag{E7.5.1}
{\text{the fixed subring $R^{G'}$ is not AS regular
for all $1\neq G'\subseteq G$.}}
\end{equation}
In this case, the smallness of $G$ can be characterized as the
property \eqref{E7.5.1}.
\end{enumerate}
\end{remark}

The next example is given in \cite[Example 2.1]{CKZ1}.

\begin{example}
\label{xxex7.6}
Let $R$ be the down-up algebra $A(0, 1)$ generated by $x$ and $y$
and subject to relations
$$x^2 y=y x^2, \quad
{\text{and}} \quad x y^2=y^2 x.$$
This is a noetherian, connected graded AS regular,
CM and PI domain. Let $H$ be the Hopf algebra
$(\Bbbk D_8)^{\circ}$ where $D_8$ is the dihedral
group of order 8. There is an $H$-action on $R$
defined as in \cite[Example 2.1]{CKZ1}. Every
Hopf subalgebra $H'$ of $H$ is of the form
$(\Bbbk (D_8/N))^{\circ}$ where $N$ is a normal
subgroup of $D_8$. By \cite[Theorem 0.1]{CKZ1},
$R^{H'}$ is not AS regular for all nontrivial
Hopf subalgebra $H'\subseteq H$. Suggested by
Remark \ref{xxrem7.5}(4), the Auslander Theorem
should hold for the $H$-action on $R$. However,
this is not true by the next paragraph, which
indicates that the ``smallness'' (of a Hopf action)
should be not defined by using the {\it failure of
the AS regularity of the fixed subrings $R^{H'}$
for all nontrivial $H'\subseteq H$.} So this is
different from the commutative case \eqref{E7.5.1}.

To see that the Auslander Theorem does not hold,
we use \cite[Lemma 2.2]{CKZ1}. Let $A:=R^{H}$. 
By \cite[Example 2.1]{CKZ1}, $A\cong\Bbbk[x,y,z,t]/(xy-zt^2)$ 
where we view $x,y,z$ of degree 4 and $t$ of degree 2. 
Hence it is AS-Gorenstein (see \cite[Theorem 0.4]{KKZ4}, for instance).
Note that each component in \cite[Lemma 2.2(1,2,3)]{CKZ1} 
is a right $A$-module, and $R$, as a right $A$-module, 
is the direct sum of all the components in \cite[Lemma 2.2]{CKZ1}. 
Hence $R=A\oplus A(-1)\oplus A(-2) \oplus M$ for some $A$-module 
$M$ since the second and the third components in 
\cite[Lemma 2.2]{CKZ1}(2,3) are isomorphic to $A(-1)$ and $A(-2)$, 
respectively.
Then $\End_A(R)$ contains nonzero elements of degree $-1$.
Thus $\End_A(R)$ can not be isomorphic to $R\# H$.
This means that the Auslander Theorem fails for this
$H$-action on $R$.

As a consequence of Theorem \ref{xxthm0.2},
$\Pty(R,H)=1$.
\end{example}

Finally we include another example which could be a test
example for the conjecture in Remark \ref{xxrem7.5}(2).

\begin{example}
\label{xxex7.7}
Let $R$ be the non-PI Sklyanin algebra of dimension at least $3$.
By \cite[Corollary 6.3]{KKZ3}, $R$ has no quasi-reflection of
finite order. By definition, every finite subgroup $G$
in $\Aut_{gr}(R)$ is c.small. Following Remark \ref{xxrem7.5}(2)
we conjecture that every such $G$ is h.small.
\end{example}

What we have learned from Theorems \ref{xxthm0.2} and \ref{xxthm2.7},
examples and comments in this section is: {\it in the
general Hopf action setting, the homological smallness
is the most reasonable replacement of the classical smallness.}



\begin{thebibliography}{99}

\bibitem[ASZ1]{ASZ1}
K. Ajitabh, S.P. Smith, and J.J. Zhang,
\emph{Auslander-Gorenstein rings},
Comm. Algebra {\bf 26}  (1998),  no. 7, 2159--2180.

\bibitem[ASZ2]{ASZ2}
K. Ajitabh, S.P. Smith, and J.J. Zhang,
\emph{Injective resolutions of some regular rings},
J. Pure Appl. Algebra {\bf 140}  (1999),  no. 1, 1--21.

\bibitem[AlC]{AlC}
J. Alev and M. Chamarie,
D{\'e}rivations et Automorphismes de Quelques Alg{\'e}bres Quantiques,
Comm. Algebra {\bf 20}(6) (1992), 1787--1802.


\bibitem[ArSZ]{ArSZ}
M. Artin, L.W. Small and J.J. Zhang,
\emph{Generic flatness for strongly Noetherian algebras}
J. Algebra {\bf 221} (1999), no. 2, 579--610.


\bibitem[Au1]{Au1}
M. Auslander,
\emph{On the purity of the branch locus},
Amer. J. Math., {\bf 84}, (1962) 116--125.

\bibitem[Au2]{Au2}
M. Auslander,
\emph{Rational singularities and almost split sequences},
Trans. Amer. Math. Soc. {\bf 293}
(1986), 511--531.


\bibitem[BHZ]{BHZ}
Y.-H. Bao, J.-W. He and J.J. Zhang,
\emph{Pertinency of Hopf actions and quotient categories
of Cohen-Mcaulay algebras},
preprint, (2016), arXiv:1603.02346.

\bibitem[BZ]{BZ}
J.P. Bell and J.J. Zhang,
\emph{Zariski Cancellation Problem for Noncommutative Algebras},
preprint, (2016), arXiv:1601.04625, 
Selecta Math. (N.S.) (in press), 
DOI: 10.1007/s00029-017-0317-7.





\bibitem[BR]{BR}
G. Benkart and T. Roby,
\emph{Down-up algebras},
J. Algebra {\bf 209} (1998), 305--344. Addendum,
J. Algebra {\bf 213} (1999), no. 1, 378.









\bibitem[CKWZ1]{CKWZ1}
K. Chan, E. Kirkman, C. Walton and J.J. Zhang,
\emph{McKay Correspondence for semisimple Hopf
actions on regular graded algebras, Part I},
preprint (2016), arXiv:1607.06977.



\bibitem[CKWZ2]{CKWZ2}
K. Chan, E. Kirkman, C. Walton and J.J. Zhang,
\emph{McKay Correspondence for semisimple Hopf
actions on regular graded algebras, Part II},
preprint (2016), arXiv:1610.01220.


\bibitem[CWWZ]{CWWZ}
K. Chan, C. Walton, Y.-H. Wang and J.J. Zhang,
\emph{Hopf actions on filtered regular algebras},
J. Algebra {\bf 397} (2014), 68--90.

\bibitem[CKZ1]{CKZ1}
J. Chen, E. Kirkman and J.J. Zhang,
\emph{Rigidity of down-up algebras with respect to
finite group coactions},
preprint, arXiv:1606.08428,
J. Pure Appl. Algebra, (in press),
2017, https://doi.org/10.1016/j.jpaa.2017.02.015.


\bibitem[CKZ2]{CKZ2}
J. Chen, E. Kirkman and J.J. Zhang,
\emph{Auslander Theorem for group coactions on 
noetherian graded down-up algebras}, 
in preparation (2017).


\bibitem[GKMW]{GKMW}
J. Gaddis, E. Kirkman, W.F. Moore and R. Won,
\emph{Auslander's Theorem for permutation actions 
on noncommutative algebras}, 
preprint (2017), arXiv:1705.00068.






\bibitem[HVZ]{HVZ}
J.-W. He, F. Van Oystaeyen and Y. Zhang,
\emph{Hopf dense Galois extensions with applications},
J. Algebra {\bf 476}, (2017), 134--160.

\bibitem[IT]{IT}
O. Iyama and R. Takahashi,
\emph{Tilting and cluster tilting for quotient singularities},
Math. Ann. {\bf 356} (2013), 1065--1105.

\bibitem[Ja]{Ja}
N. Jacobson, 
\emph{A note on Lie algebras of characteristic $p$}, 
Amer. J. Math. {\bf 74}, (1952). 357--359.

\bibitem[JZ]{JZ}
N. Jing and J.J. Zhang,
\emph{On the trace of graded automorphisms},
J. Algebra {\bf 189} (1997), no. 2, 353--376.



\bibitem[Ki]{Ki}
E. Kirkman, Private communication, (2015).

\bibitem[KK]{KK}
E. Kirkman and J. Kuzmanovich,
\emph{Fixed subrings of Noetherian graded regular rings},
J. Algebra {\bf 288} (2005), no. 2, 463--484.


\bibitem[KKZ1]{KKZ1}
E. Kirkman, J. Kuzmanovich and J.J. Zhang,
\emph{Invariant theory of finite group actions on down-up algebras},
Transform. Groups {\bf 20} (2015), no. 1, 113--165.


\bibitem[KKZ2]{KKZ2}
E. Kirkman, J. Kuzmanovich and J.J. Zhang,
\emph{Invariants of $(-1)$-skew polynomial rings
under permutation representations},
Recent advances in representation theory, quantum groups,
algebraic geometry, and related topics,  155--192, Contemp. Math.,
{\bf 623}, Amer. Math. Soc., Providence, RI, 2014.

\bibitem[KKZ3]{KKZ3}
E. Kirkman, J. Kuzmanovich and J.J. Zhang,
\emph{Rigidity of graded regular algebras},
Trans. Amer. Math. Soc. {\bf 360} (2008), no. 12, 6331--6369.

\bibitem[KKZ4]{KKZ4}
E. Kirkman, J. Kuzmanovich and J.J. Zhang,
\emph{Noncommutative complete intersections},
J. Algebra {\bf 429} (2015), 253--286.



\bibitem[KMP]{KMP}
E. Kirkman, I. Musson and D. Passman,
\emph{Noetherian down-up algebras},
Proc. Amer. Math. Soc. {\bf 127} (1999), 3161--3167.


\bibitem[KL]{KL}
G. R. Krause and T. H. Lenagan,
\emph{Growth of algebras and Gelfand-Kirillov dimension},
Research Notes in Mathematics, Pitman Adv. Publ. Program, vol.
{\bf 116} (1985).




\bibitem[MR]{MR}
J.C. McConnell and J.C . Robson,
\emph{Noncommutative Noetherian Rings}, Wiley, Chichester, 1987.



\bibitem[Mon]{Mon}
S. Montgomery,
\emph{Hopf Algebras and Their Actions on Rings},
CBMS Reg. Conf. Ser. Math. 82, Amer. Math. Soc., Providence, RI, 1993.


\bibitem[Mor]{Mor}
I. Mori,
\emph{McKay-type correspondence for AS-regular algebras},
J. Lond. Math. Soc. (2) {\bf 88}  (2013),  no. 1, 97--117.

\bibitem[MU]{MU}
I. Mori and K. Ueyama,
\emph{Ample Group Action on AS-regular Algebras and Noncommutative
Graded Isolated Singularities},
Trans. Amer. Math. Soc., {\bf 368} (2016), no. 10, 7359--7383.






\bibitem[Ro]{Ro}
J.J. Rotman,
\emph{An introduction to homological algebra},
Pure and Applied Mathematics, {\bf 85}.
Academic Press, Inc. New York-London, 1979.




\bibitem[SZ]{SZ}
J.T. Stafford and J.J. Zhang,
\emph{Homological properties of (graded) Noetherian PI rings},
J. Algebra {\bf 168}  (1994),  no. 3, 988--1026.


\bibitem[Ue]{Ue}
K. Ueyama,
\emph{Graded maximal Cohen-Macaulay modules over noncommutative graded
Gorenstein isolated singularities},
J. Algebra {\bf 383}  (2013), 85--103.




\bibitem[YZ]{YZ}
A. Yekutieli and J.J. Zhang,
\emph{Rings with Auslander dualizing complexes},
J. Algebra {\bf 213}  (1999),  no. 1, 1--51.


\bibitem[Zh1]{Zh1}
J.J. Zhang,
\emph{Connected graded Gorenstein algebras with enough normal elements},
J. Algebra {\bf 189}  (1997),  no. 2, 390--405.


\bibitem[Zh2]{Zh2}
J.J. Zhang,
\emph{Twisted graded algebras and equivalences of graded categories},
Proc. London Math. Soc. (3) {\bf 72} (1996), no. 2, 281--311.


\end{thebibliography}

\end{document}